\documentclass[10pt,leqno]{amsart}
\usepackage{amssymb}
\usepackage{xypic}
\begin{document}
\theoremstyle{plain}
\newtheorem{thm}{Theorem}[section]
\newtheorem*{thm1}{Theorem 1}
\newtheorem*{thm2}{Theorem 2}
\newtheorem{lemma}[thm]{Lemma}
\newtheorem{lem}[thm]{Lemma}
\newtheorem{cor}[thm]{Corollary}
\newtheorem{prop}[thm]{Proposition}
\newtheorem{propose}[thm]{Proposition}
\newtheorem{variant}[thm]{Variant}
\theoremstyle{definition}
\newtheorem{notations}[thm]{Notations}
\newtheorem{rem}[thm]{Remark}
\newtheorem{rmk}[thm]{Remark}
\newtheorem{rmks}[thm]{Remarks}
\newtheorem{defn}[thm]{Definition}
\newtheorem{ex}[thm]{Example}
\newtheorem{claim}[thm]{Claim}
\newtheorem{ass}[thm]{Assumption}
\numberwithin{equation}{section}
\newcounter{elno}                
\def\points{\list
{\hss\llap{\upshape{(\roman{elno})}}}{\usecounter{elno}}} 
\let\endpoints=\endlist


\catcode`\@=11
%
%
\def\opn#1#2{\def#1{\mathop{\kern0pt\fam0#2}\nolimits}} 
\def\bold#1{{\bf #1}}%
\def\underrightarrow{\mathpalette\underrightarrow@}
\def\underrightarrow@#1#2{\vtop{\ialign{$##$\cr
 \hfil#1#2\hfil\cr\noalign{\nointerlineskip}%
 #1{-}\mkern-6mu\cleaders\hbox{$#1\mkern-2mu{-}\mkern-2mu$}\hfill
 \mkern-6mu{\to}\cr}}}
\let\underarrow\underrightarrow
\def\underleftarrow{\mathpalette\underleftarrow@}
\def\underleftarrow@#1#2{\vtop{\ialign{$##$\cr
 \hfil#1#2\hfil\cr\noalign{\nointerlineskip}#1{\leftarrow}\mkern-6mu
 \cleaders\hbox{$#1\mkern-2mu{-}\mkern-2mu$}\hfill
 \mkern-6mu{-}\cr}}}
%
%

%
\def\:{\colon}
\let\oldtilde=\tilde
\def\tilde#1{\mathchoice{\widetilde{#1}}{\widetilde{#1}}%
{\indextil{#1}}{\oldtilde{#1}}}
\def\indextil#1{\lower2pt\hbox{$\textstyle{\oldtilde{\raise2pt%
\hbox{$\scriptstyle{#1}$}}}$}}
\def\pnt{{\raise1.1pt\hbox{$\textstyle.$}}}
%

%
\let\amp@rs@nd@\relax
\newdimen\ex@\ex@.2326ex
\newdimen\bigaw@
\newdimen\minaw@
\minaw@16.08739\ex@
\newdimen\minCDaw@
\minCDaw@2.5pc
\newif\ifCD@
\def\minCDarrowwidth#1{\minCDaw@#1}
\newenvironment{CD}{\@CD}{\@endCD}
\def\@CD{\def\A##1A##2A{\llap{$\vcenter{\hbox
 {$\scriptstyle##1$}}$}\Big\uparrow\rlap{$\vcenter{\hbox{%
$\scriptstyle##2$}}$}&&}%
\def\V##1V##2V{\llap{$\vcenter{\hbox
 {$\scriptstyle##1$}}$}\Big\downarrow\rlap{$\vcenter{\hbox{%
$\scriptstyle##2$}}$}&&}%
\def\={&\hskip.5em\mathrel
 {\vbox{\hrule width\minCDaw@\vskip3\ex@\hrule width
 \minCDaw@}}\hskip.5em&}%
\def\verteq{\Big\Vert&&}%
\def\noarr{&&}%
\def\vspace##1{\noalign{\vskip##1\relax}}\relax\let\amp@rs@nd@&\iffalse}\fi
 \CD@true\vcenter\bgroup\relax\let\\=\cr\iffalse}\fi\tabskip\z@skip\baselineskip20\ex@
 \lineskip3\ex@\lineskiplimit3\ex@\halign\bgroup
 &\hfill$\m@th##$\hfill\cr}
\def\@endCD{\cr\egroup\egroup}
%
\def\>#1>#2>{\amp@rs@nd@\setbox\z@\hbox{$\scriptstyle
 \;{#1}\;\;$}\setbox\@ne\hbox{$\scriptstyle\;{#2}\;\;$}\setbox\tw@
 \hbox{$#2$}\ifCD@
 \global\bigaw@\minCDaw@\else\global\bigaw@\minaw@\fi
 \ifdim\wd\z@>\bigaw@\global\bigaw@\wd\z@\fi
 \ifdim\wd\@ne>\bigaw@\global\bigaw@\wd\@ne\fi
 \ifCD@\hskip.5em\fi
 \ifdim\wd\tw@>\z@
 \mathrel{\mathop{\hbox to\bigaw@{\rightarrowfill}}\limits^{#1}_{#2}}\else
 \mathrel{\mathop{\hbox to\bigaw@{\rightarrowfill}}\limits^{#1}}\fi
 \ifCD@\hskip.5em\fi\amp@rs@nd@}
\def\<#1<#2<{\amp@rs@nd@\setbox\z@\hbox{$\scriptstyle
 \;\;{#1}\;$}\setbox\@ne\hbox{$\scriptstyle\;\;{#2}\;$}\setbox\tw@
 \hbox{$#2$}\ifCD@
 \global\bigaw@\minCDaw@\else\global\bigaw@\minaw@\fi
 \ifdim\wd\z@>\bigaw@\global\bigaw@\wd\z@\fi
 \ifdim\wd\@ne>\bigaw@\global\bigaw@\wd\@ne\fi
 \ifCD@\hskip.5em\fi
 \ifdim\wd\tw@>\z@
 \mathrel{\mathop{\hbox to\bigaw@{\leftarrowfill}}\limits^{#1}_{#2}}\else
 \mathrel{\mathop{\hbox to\bigaw@{\leftarrowfill}}\limits^{#1}}\fi
 \ifCD@\hskip.5em\fi\amp@rs@nd@}
%
%
\newenvironment{CDS}{\@CDS}{\@endCDS}
\def\@CDS{\def\A##1A##2A{\llap{$\vcenter{\hbox
 {$\scriptstyle##1$}}$}\Big\uparrow\rlap{$\vcenter{\hbox{%
$\scriptstyle##2$}}$}&}%
\def\V##1V##2V{\llap{$\vcenter{\hbox
 {$\scriptstyle##1$}}$}\Big\downarrow\rlap{$\vcenter{\hbox{%
$\scriptstyle##2$}}$}&}%
\def\={&\hskip.5em\mathrel
 {\vbox{\hrule width\minCDaw@\vskip3\ex@\hrule width
 \minCDaw@}}\hskip.5em&}
\def\verteq{\Big\Vert&}
\def\novarr{&}
\def\noharr{&&}
\def\SE##1E##2E{\slantedarrow(0,18)(4,-3){##1}{##2}&}
\def\SW##1W##2W{\slantedarrow(24,18)(-4,-3){##1}{##2}&}
\def\NE##1E##2E{\slantedarrow(0,0)(4,3){##1}{##2}&}
\def\NW##1W##2W{\slantedarrow(24,0)(-4,3){##1}{##2}&}
\def\slantedarrow(##1)(##2)##3##4{%
\thinlines\unitlength1pt\lower 6.5pt\hbox{\begin{picture}(24,18)%
\put(##1){\vector(##2){24}}%
\put(0,8){$\scriptstyle##3$}%
\put(20,8){$\scriptstyle##4$}%
\end{picture}}}
\def\vspace##1{\noalign{\vskip##1\relax}}\relax\let\amp@rs@nd@&\iffalse}\fi
 \CD@true\vcenter\bgroup\relax\let\\=\cr\iffalse}\fi\tabskip\z@skip\baselineskip20\ex@
 \lineskip3\ex@\lineskiplimit3\ex@\halign\bgroup
 &\hfill$\m@th##$\hfill\cr}
\def\@endCDS{\cr\egroup\egroup}
%
\newdimen\TriCDarrw@
\newif\ifTriV@
\newenvironment{TriCDV}{\@TriCDV}{\@endTriCD}
\newenvironment{TriCDA}{\@TriCDA}{\@endTriCD}
\def\@TriCDV{\TriV@true\def\TriCDpos@{6}\@TriCD}
\def\@TriCDA{\TriV@false\def\TriCDpos@{10}\@TriCD}
\def\@TriCD#1#2#3#4#5#6{%
\setbox0\hbox{$\ifTriV@#6\else#1\fi$}
\TriCDarrw@=\wd0 \advance\TriCDarrw@ 24pt
\advance\TriCDarrw@ -1em
\def\SE##1E##2E{\slantedarrow(0,18)(2,-3){##1}{##2}&}
\def\SW##1W##2W{\slantedarrow(12,18)(-2,-3){##1}{##2}&}
\def\NE##1E##2E{\slantedarrow(0,0)(2,3){##1}{##2}&}
\def\NW##1W##2W{\slantedarrow(12,0)(-2,3){##1}{##2}&}
\def\slantedarrow(##1)(##2)##3##4{\thinlines\unitlength1pt
\lower 6.5pt\hbox{\begin{picture}(12,18)%
\put(##1){\vector(##2){12}}%
\put(-4,\TriCDpos@){$\scriptstyle##3$}%
\put(12,\TriCDpos@){$\scriptstyle##4$}%
\end{picture}}}
\def\={\mathrel {\vbox{\hrule
   width\TriCDarrw@\vskip3\ex@\hrule width
   \TriCDarrw@}}}
\def\>##1>>{\setbox\z@\hbox{$\scriptstyle
 \;{##1}\;\;$}\global\bigaw@\TriCDarrw@
 \ifdim\wd\z@>\bigaw@\global\bigaw@\wd\z@\fi
 \hskip.5em
 \mathrel{\mathop{\hbox to \TriCDarrw@
{\rightarrowfill}}\limits^{##1}}
 \hskip.5em}
\def\<##1<<{\setbox\z@\hbox{$\scriptstyle
 \;{##1}\;\;$}\global\bigaw@\TriCDarrw@
 \ifdim\wd\z@>\bigaw@\global\bigaw@\wd\z@\fi
 \mathrel{\mathop{\hbox to\bigaw@{\leftarrowfill}}\limits^{##1}}
 }
 \CD@true\vcenter\bgroup\relax\let\\=\cr\iffalse}\fi
 \tabskip\z@skip\baselineskip20\ex@
 \lineskip3\ex@\lineskiplimit3\ex@
 \ifTriV@
 \halign\bgroup
 &\hfill$\m@th##$\hfill\cr
#1&\multispan3\hfill$#2$\hfill&#3\\
&#4&#5\\
&&#6\cr\egroup%
\else
 \halign\bgroup
 &\hfill$\m@th##$\hfill\cr
&&#1\\%
&#2&#3\\
#4&\multispan3\hfill$#5$\hfill&#6\cr\egroup
\fi}
\def\@endTriCD{\egroup} 
\newcommand{\mc}{\mathcal} 
\newcommand{\mb}{\mathbb} 
\newcommand{\surj}{\twoheadrightarrow} 
\newcommand{\inj}{\hookrightarrow} \newcommand{\zar}{{\rm zar}} 
\newcommand{\an}{{\rm an}} \newcommand{\red}{{\rm red}} 
\newcommand{\Rank}{{\rm rk}} \newcommand{\codim}{{\rm codim}} 
\newcommand{\rank}{{\rm rank}} \newcommand{\Ker}{{\rm Ker \ }} 
\newcommand{\Pic}{{\rm Pic}} \newcommand{\Div}{{\rm Div}} 
\newcommand{\Hom}{{\rm Hom}} \newcommand{\im}{{\rm im}} 
\newcommand{\Spec}{{\rm Spec \,}} \newcommand{\Sing}{{\rm Sing}} 
\newcommand{\sing}{{\rm sing}} \newcommand{\reg}{{\rm reg}} 
\newcommand{\Char}{{\rm char}} \newcommand{\Tr}{{\rm Tr}} 
\newcommand{\Gal}{{\rm Gal}} \newcommand{\Min}{{\rm Min \ }} 
\newcommand{\Max}{{\rm Max \ }} \newcommand{\Alb}{{\rm Alb}\,} 
\newcommand{\GL}{{\rm GL}\,} 
\newcommand{\ie}{{\it i.e.\/},\ } \newcommand{\niso}{\not\cong} 
\newcommand{\nin}{\not\in} 
\newcommand{\soplus}[1]{\stackrel{#1}{\oplus}} 
\newcommand{\by}[1]{\stackrel{#1}{\rightarrow}} 
\newcommand{\longby}[1]{\stackrel{#1}{\longrightarrow}} 
\newcommand{\vlongby}[1]{\stackrel{#1}{\mbox{\large{$\longrightarrow$}}}} 
\newcommand{\ldownarrow}{\mbox{\Large{\Large{$\downarrow$}}}} 
\newcommand{\lsearrow}{\mbox{\Large{$\searrow$}}} 
\renewcommand{\d}{\stackrel{\mbox{\scriptsize{$\bullet$}}}{}} 
\newcommand{\dlog}{{\rm dlog}\,} 
\newcommand{\longto}{\longrightarrow} 
\newcommand{\vlongto}{\mbox{{\Large{$\longto$}}}} 
\newcommand{\limdir}[1]{{\displaystyle{\mathop{\rm lim}_{\buildrel\longrightarrow\over{#1}}}}\,} 
\newcommand{\liminv}[1]{{\displaystyle{\mathop{\rm lim}_{\buildrel\longleftarrow\over{#1}}}}\,} 
\newcommand{\norm}[1]{\mbox{$\parallel{#1}\parallel$}} 
\newcommand{\boxtensor}{{\Box\kern-9.03pt\raise1.42pt\hbox{$\times$}}} 
\newcommand{\into}{\hookrightarrow} \newcommand{\image}{{\rm image}\,} 
\newcommand{\Lie}{{\rm Lie}\,} 
\newcommand{\CM}{\rm CM}
\newcommand{\sext}{\mbox{${\mathcal E}xt\,$}} 
\newcommand{\shom}{\mbox{${\mathcal H}om\,$}} 
\newcommand{\coker}{{\rm coker}\,} 
\newcommand{\sm}{{\rm sm}} 
\newcommand{\tensor}{\otimes} 
\renewcommand{\iff}{\mbox{ $\Longleftrightarrow$ }} 
\newcommand{\supp}{{\rm supp}\,} 
\newcommand{\ext}[1]{\stackrel{#1}{\wedge}} 
\newcommand{\onto}{\mbox{$\,\>>>\hspace{-.5cm}\to\hspace{.15cm}$}} 
\newcommand{\propsubset} {\mbox{$\textstyle{ 
\subseteq_{\kern-5pt\raise-1pt\hbox{\mbox{\tiny{$/$}}}}}$}} 
\newcommand{\sA}{{\mathcal A}} 
\newcommand{\sB}{{\mathcal B}} \newcommand{\sC}{{\mathcal C}} 
\newcommand{\sD}{{\mathcal D}} \newcommand{\sE}{{\mathcal E}} 
\newcommand{\sF}{{\mathcal F}} \newcommand{\sG}{{\mathcal G}} 
\newcommand{\sH}{{\mathcal H}} \newcommand{\sI}{{\mathcal I}} 
\newcommand{\sJ}{{\mathcal J}} \newcommand{\sK}{{\mathcal K}} 
\newcommand{\sL}{{\mathcal L}} \newcommand{\sM}{{\mathcal M}} 
\newcommand{\sN}{{\mathcal N}} \newcommand{\sO}{{\mathcal O}} 
\newcommand{\sP}{{\mathcal P}} \newcommand{\sQ}{{\mathcal Q}} 
\newcommand{\sR}{{\mathcal R}} \newcommand{\sS}{{\mathcal S}} 
\newcommand{\sT}{{\mathcal T}} \newcommand{\sU}{{\mathcal U}} 
\newcommand{\sV}{{\mathcal V}} \newcommand{\sW}{{\mathcal W}} 
\newcommand{\sX}{{\mathcal X}} \newcommand{\sY}{{\mathcal Y}} 
\newcommand{\sZ}{{\mathcal Z}} \newcommand{\ccL}{\sL} 
 \newcommand{\A}{{\mathbb A}} \newcommand{\B}{{\mathbb 
B}} \newcommand{\C}{{\mathbb C}} \newcommand{\D}{{\mathbb D}} 
\newcommand{\E}{{\mathbb E}} \newcommand{\F}{{\mathbb F}} 
\newcommand{\G}{{\mathbb G}} \newcommand{\HH}{{\mathbb H}} 
\newcommand{\I}{{\mathbb I}} \newcommand{\J}{{\mathbb J}} 
\newcommand{\M}{{\mathbb M}} \newcommand{\N}{{\mathbb N}} 
\renewcommand{\P}{{\mathbb P}} \newcommand{\Q}{{\mathbb Q}} 

\newcommand{\R}{{\mathbb R}} \newcommand{\T}{{\mathbb T}} 
\newcommand{\U}{{\mathbb U}} \newcommand{\V}{{\mathbb V}} 
\newcommand{\W}{{\mathbb W}} \newcommand{\X}{{\mathbb X}} 
\newcommand{\Y}{{\mathbb Y}} \newcommand{\Z}{{\mathbb Z}} 
\title[Hilbert-Kunz density function and Hilbert-Kunz multiplicity] 
{Hilbert-Kunz density function and Hilbert-Kunz multiplicity} 
\author{V. Trivedi} 
\address{School of Mathematics, Tata Institute of Fundamental Research, 
Homi Bhabha Road, Mumbai-400005, India} 
\email{vija@math.tifr.res.in} 
\date{} 

\begin{abstract} For a pair $(M, I)$, where $M$ is finitely generated
 graded module over a standard graded ring $R$ of dimension $d$, and $I$
 is a graded ideal with $\ell(R/I) < \infty$, we introduce a new invariant
 $HKd(M, I)$ called the {\em Hilbert-Kunz density function}. 
In Theorem~\ref{t1}, we relate this to the Hilbert-Kunz multiplicity 
$e_{HK}(M, I)$ by an integral formula.  

We prove that the Hilbert-Kunz density function is additive. Moreover it
satisfies a multiplicative formula for a Segre product of rings. This gives a 
formula for $e_{HK}$ of the Segre product of rings in terms of the HKd of 
the rings involved. As a corollary, $e_{HK}$ of  
the Segre product of any finite number of  projective
 curves is a rational number. 
 \end{abstract}

\maketitle \section{Introduction}

Let $R$ be a  Noetherian ring of prime characteristic
$p >0$ and of dimension $d$ and let
$I\subseteq R$ be an ideal of finite colength. Then we recall that
the Hilbert-Kunz
multiplicity of $R$ with
respect to $I$ is defined as
$$e_{HK}(R, I) = \lim_{n\to \infty}\frac{\ell(R/I^{[p^n]})}{p^{nd}},$$
where
$ I^{[p^n]} = n$-th Frobenius power of $I$
= the ideal generated by $p^n$-th powers of elements of $I$. This   
 is an ideal of finite colength and  $\ell(R/I^{[p^n]})$ denotes the
length of
the
$R$-module $R/I^{[p^n]}$. Existence of the limit was proved by 
Monsky [M]. Though this invariant has been extensively 
studied, over the years (see the survey article [Hu]), it has been difficult to
handle  (even in the graded case) as various standard techniques,
 used for studying multiplicities, are not applicable for the invariant 
$e_{HK}$.

Here  we introduce a new invariant for a
  pair $(M, I)$, where  $R$ 
 is a Noetherian standard graded ring  of dimension $d$ over a perfect field
 $k$ of $\Char~p >0$, $I$ is 
a homogeneous ideal of $R$ such that $\ell(R/I)< \infty$, and $M$   
 is  a finitely generated non-negatively graded
 $R$-module.

This invariant for a pair $(M, I)$, which we call the
 {\it Hilbert-Kunz density function} of $(M, I)$, is a compactly supported 
function
$$HKd(M, I):\R\longrightarrow\R,$$ given by 
$$HKd(M, I)(x) = f(x) = \lim_{n\to \infty}g_n(x),$$ 
 where  $g_n:\R\rightarrow \R$ is given in Notations~\ref{n2}.
We show that this limit makes sense and in fact 
$$HKd(M, I)(x) = f(x) = \lim_{n\to \infty}g_n(x) = \lim_{n\to 
\infty}f_n(x),$$
where $f_n(x) = (1/q^{d-1})\ell(M/I^{[q]}M)_{\lfloor xq\rfloor}$.
More precisely we prove the following theorem, which also
relates the Hilbert-Kunz multiplicity with the Hilbert-Kunz 
density function.

\begin{thm}\label{t1}If $R$ is of dimension $\geq 2$ then 
each $g_n:\R\longto \R$ is a compactly 
supported, piecewise linear continuous function such that
$\{g_n\}_{n\in \N}$ is a uniformly convergent sequence. If 
$\lim_{n\to \infty} g_n(x) = f(x)$, then $f(x)$ is a compactly 
supported continuous function, and  
$$e_{HK}(M, I) = \int_{\R}f(x)dx. $$\end{thm}

We note that the HK density function 
plays the same role as $e_{HK}$ vis-a-vis tight closure, in the graded setup
 (see Remark~\ref{r7}). Also like the $e_{HK}$ multiplicity, the HK density function is 
additive (Proposition~\ref{r77}).

One of the remarkable properties of the HK density function (which also makes
 computations of $e_{HK}$ in various cases simpler, and
 makes them possible in many new cases) is that 
it is `multiplicative' for Segre products. 

In Proposition~\ref{p2} we state and prove  
this  {\it multiplicative formula}.
In particular, we prove the following:

\vspace{10pt}
 
\noindent{\bf Proposition}\quad {\it If   $(R, I)$ and $(S, J)$ are two pairs as above, and if
  $HKd(R, I) = f$ 
and $HKd(S, J) = g$ 
with $\dim~R = d_1$ and $\dim~S = d_2$ then their Segre product satisfies:
$$HKd (R\#S, I\# J)(x)  =  \frac{e_0(R)}{(d_1-1)!}x^{d_1-1}g(x)+
\frac{e_0(S)}{(d_2-1)!}x^{d_2-1}f(x) - f(x)g(x),$$}
Here  $e_0(R)$ denotes the Hilbert-Samuel multiplicity of $R$ with respect
 to its irrelevant maximal ideal.

\vspace{5pt}

This implies that $e_{HK}$ of any finite  Segre products of rings can be written
 in terms of the $HKd$ functions of the rings involved, whereas 
Example~\ref{ex2} suggests that any such `multiplicative formula'
does not hold for
 HK multiplicities.

In Section~3  we compute $HKd(R, I)$, for projective spaces and nonsingular
 projective curves (and  hence of arbitrary Segre products of these). 
Theorem~\ref{t1} then yields formulas for HK multiplicities. We note that 
the HK multiplicity  of a product of $\P^n\times \P^m$ was 
known earlier ([EY]).

In the case of a nonsingular projective curve $X = {\rm Proj}~R$ of
 degree $d$, we can associate its {\it HN data} of a set of rational numbers
$(d, \{r_i\}_i, \{a_i\}_i)$, where 
  (see [B], [T1], for the corresponding study of the HK multiplicity
  in this context)
$\{r_i\}_i$ and
$\{a_i\}_i$ denote, respectively, the ranks and normalized strong
 Harder-Narasimhan
 slopes 
of the associated syzygy bundle $V$ on $X$ {see (\ref{e14}) for details). Then 
 it turns out that  the  density function $HKd(R, {\bf m})$,
 is a piecewise linear
polynomial with rational coefficients, and with 
 points of 
singularites ({\it i.e.}, non-smoothness) precisely at the points
 $\{1-(a_i/d)\}_i$. Moreover $d$ and  and  the set $\{r_i\}_i$ can
 also be easily recovered 
from the density function (see Example~~\ref{ex1}).

 This implies that (since $HKd(R, {\bf m})$ and hence)
the numbers $\{r_i\}$, $\{a_i\}$ are {\it intrinsic invariants} of the pair
$(R, {\bf m})$.

Now,  by Proposition~\ref{p2}, 
the HK density function of a Segre product of $n$ projective 
curves $\{X_j\}_j$, corresponding to the pairs $(R_j, I_j)_j$, is a piecewise degree
 $n$-polynomial with rational coefficients, with the set of singular  points 
$\subseteq  \{1-a_{ij}/{\tilde d_j},d_{ij} \}_i$, where $\{a_{ij}\}_i$ are the
 normalized
strong  HN slopes and ${\tilde d_j}$ is the 
degree of the curves $X_j$ and $d_{ij}$ 
are the degrees of the chosen generators of the ideals $I_j$.
 Hence, by Theorem~\ref{t1},  we
 deduce, as a corollary, 

{\it The  HK multiplicity of the Segre product of any finite number of 
projective curves  is a rational number}.

In Example~\ref{ex2} we write down  
the Hilbert-Kunz density function for the Segre product   
of two  dimensional rings  $(R, {\bf m}_1)$ and $(S,{\bf m}_2)$.
 If  $(d, \{r_i\}_i, \{a_i\}_i)$ and 
$(g, \{s_j\}_j, \{b_j\}_j)$ are the datum associated to 
the pairs $(R, {\bf m}_1)$ and 
and   $(S, {\bf m}_2)$ respectively
then  we deduce that  
 $e_{HK}(R\#S, {\bf m}_1\# {\bf m}_2)$ is a polynomial 
in $\{r_i, a_i/d\}_i$ and
 $\{s_j, b_j/g\}_j$ but the formula for it 
depends on the  
 relative
 positions of the $a_i/d$ and $b_j/g$ on the real line.
 On the other hand 
we know (see [B], [T1]) that
$e_{HK}(R, {\bf m}_1) = d + \sum_ir_ia_i^2/d$ and
 $e_{HK}(S, {\bf m}_2) = g + \sum_js_jb_j^2/g $.

This 
 suggests that unlike the functions such as  multiplicity and 
$HKd$ function,  {\it $e_{HK}$ of a Segre product of rings cannot be determined
in terms of the  $e_{HK}$ of the individual rings alone.}

Overall it seems that 
$HKd$ is relatively easier to calculate (as one is computing a  
`limit' of each graded piece rather than computing a limit of a sum of graded pieces)
on the other hand it carries  more information ({\it e.g.} in the case 
of projective curves, the normalized slopes $\{a_i/d\}$ are precisely the 
 points of singularities of the $HKd$, and 
$\{r_i\}$ also  are  recoverable  from the density function).

 In [T2], we give another application of HK density functions
 to give an approach to $e_{HK}$ in characteristic $0$.

We expect the techniques introduced in this paper to have several other 
interesting  applications as well.

For example in a forthcoming paper [Ma],
 it is shown that the HK density function of a tensor product of 
standard graded rings equals the {\it convolution} of the HK density 
of the factors. 

Recall that the set of compactly supported continuous functions
 $f:\R\longto \R$ are 
in bijective correspondence with the set of their holomorphic 
Fourier transforms ${\hat f}$, where 
${\hat f}(t) = \int_{\R}f(x)e^{-itx}dx$, for $t\in \C$.  
Since HK density functions are compactly supported functions on $\R$, 
 for a pair $(M, I)$, the  HK density function $f = HKd(M, I)$
 corresponds to  
its Fourier transform ${\hat f}$ and moreover ${\hat f}(0) = e_{HK}(M, I)$.
  We also know that the Fourier transform of the 
convolution of two such functions is the pointwise product of their Fourier 
transforms.

In particular, the results in the present paper 
suggest possible applications of techniques from harmonic analysis in the 
study of HK multiplicities; we hope to return to this later. 

One can also ask the following 

\vspace{5pt}

\noindent{\bf Question}\quad 
Can this notion of HK density function be generalized to a 
Noetherian local ring $(R, {\bf m})$, with respect to the ${\bf m}$-adic 
filtration?   

\vspace{5pt}

The paper is organised as follows.
In the second section we prove the main existence theorem, namely Theorem~\ref{t1}.
 In  Lemma~\ref{l2} (which is the heart of the theorem), we prove that the  cohomologies of $n^{th}$ Frobenius 
pull back
 of a locally free sheaf (as in given in  Equation~(\ref{e15}))
twisted by $\sQ(m)$ ($\sQ$ is a coherent sheaf of
 dimension ${\bar d}$) is bounded by a polynomial in $m$, $p^n$
of degree ${\bar d}$ with  invariants of $\sQ$ as the coefficients.

The main theorem is inspired by the   
  philosophy espoused in the survey article [Hu] that
{\it  the map from $R$ to $R^{1/p}$ is essentially the map from $R^{1/q}$ 
to $R^{1/qp}$} (for this we  state and prove a sheaf 
theoretic version in Lemma~\ref{l4}).  

We also look at the case of  dimension $1$ in Theorem~\ref{t*1}, and note that, 
for each $x$, the sequence of functions 
$g_n(x)$ converges pointwise to $f(x)$ but need not converge  uniformly.
 However $\int_{\R}f(x)dx$ still gives the HK multiplicity. 

\vspace{5pt}

The author  thanks the referee for pointing out  that HK density 
function is additive, and also for providing various suggestions which 
greatly improved the exposition of the result. 
.

\section{Main existence theorem}
Throughout the paper, $R$ 
 is a Noetherian standard graded ring  of dimension $d$ over a
perfect field 
$k$ of $\Char~p >0$, $I$ is 
an homogeneous ideal of $R$ such that $\ell(R/I)< \infty$, and $M$   
 is  a finitely generated non-negatively graded
 $R$-module.

\begin{notations}\label{n2}For the pair $(M, I)$ we define a sequence of functions 
$\{g_n:\R\rightarrow \R\}$, as follows: 
Fix $n\in \N$ and denote $q = p^n$.
Let $ x\in \R$ then $x\in [m/q, (m+1)/q)$, for some integer $m$. If
 $x = m/q$ then define 
$$g_n(x) = 1/q^{d-1}\ell(M/I^{[q]}M)_{m},$$

Otherwise, we can write $x= (1-t)m/q + t(m+1)/q$,  
for some unique $t\in [0, 1)$, and then 
we define $$g_n(x) = (1-t)g_n(m/q) + tg_n((m+1)/q).$$ 
Let $\mu \geq \mu(I)$ be a fixed number, where $\mu(I)$ is the minimal 
number of generators of the ideal $I$. 
\end{notations}
\begin{lemma}\label{l1}Each $g_n $ is a compactly supported continuous 
function. Moreover there is a fixed compact set containing 
$\cup_n\mbox{supp}~g_n$.
\end{lemma}
\begin{proof}The continuity property is obvious.
Let $n_0\in \N$ such that ${\bf m}^{n_0} \subseteq I$, where ${\bf m}$ is the graded 
maximal ideal. Therefore, for 
$m\geq n_0\mu q $, we have $R_m \subseteq ({\bf m}^{n_0})^{\mu q}
\subseteq I^{\mu q} \subseteq I^{[q]}$.
Let $l$ be a positive integer such that 
$R_mM_l = M_{m+l}$, for $m\geq 0$. Then for $m\geq n_0\mu q+l$, we have 
$$M_m = R_{m-l}M_l \subseteq ({\bf m}^{n_0})^{\mu q}M_l \subseteq I^{\mu q}M_l
 \subseteq 
I^{[q]}M_l.$$ This implies
$\ell(M/I^{[q]}M)_{m} = 0$, if $m \geq l+ n_0\mu q$.
Therefore support of $g_n \subseteq [0, (n_0\mu)+l/q]$. \end{proof}

\vspace{10pt}
\begin{rmk}\label{r11} Since replacing $R$ and $M$ by $R\tensor_k{\bar k}$  and 
$M\tensor_k {\bar k}$, the function $g_n:\R\longto \R$ remains unchanged, we can 
asssume without loss of generality that $k$ is algebraically closed.
\end{rmk}

Henceforth we assume that $R$ is a standard graded  ring of dimension $\geq 2$ (unless otherwise stated).
Let $I$ be generated by homogeneous elements, say $h_1,\ldots, h_\mu$ of 
positive  
degrees $d_1, \ldots, d_\mu$ respectively.
 Let $X = {\rm Proj}~R$; then we have an associated  canonical exact
 sequence of locally free sheaves of $\sO_X$-modules
(moreover the sequence is locally split exact). Due to Remark~\ref{r11}, we can 
also assume $k$ is an 
algebraically closed field.

\begin{equation}\label{e2}
0\longto V\longto \oplus_i\sO_X(1-d_i)\longto \sO_X(1)\longto 0,\end{equation}
where $\sO_X(1-d_i)\longto \sO_X(1)$ is given by the multiplication by the 
 element  $h_i$.

For a coherent sheaf $\sQ$ of $\sO_X$-modules we have  
 a long exact sequence of cohomologies
\begin{equation}\label{*}0\longto H^0(X, F^{n*}V\tensor \sQ(m))\longto
 \oplus_iH^0(X, \sQ(q-qd_i+m))\by{\phi_{m,q}(\sQ)} 
H^0(X, \sQ(q+m))\end{equation}

$$ \longto H^1(X, F^{n*}V\tensor\sQ(m))\longto \cdots,$$
for $m\geq 0$, $n\geq 0$ and $q = p^n$. (Here $F^n:X\longto X$ is the $n^{th}$ 
iterated Frobenius map).

\vspace{10pt}

We fix a set of notations used throughout the paper.

\begin{notations}\label{n1}Let $Q = \oplus_{m\geq 0}Q_m$ be a  nonnegatively 
graded finitely generated $R$-module and  let $\sQ$ be the associated coherent 
sheaf of 
$\sO_X$-modules. Therefore
 $Q_m = H^0(X,\sQ(m))$, for $m>>0$. 
\begin{enumerate}

\item  ${\tilde m}\geq 1$
 is the least integer  such that,
$$Q_{m+1} = R_{1}Q_{m}, ~~~\mbox{and}~~~Q_m
 = H^0(X, \sQ(m))~~~\mbox{and}~~~ h^i(X, \sQ(m-i))
 = 0,$$
for all $m\geq {\tilde m}$ and
 for all $i\geq 1$, 
\item ${\bar d} =$  the dimension of the support of $\sQ$ as a sheaf of
$\sO_X$-modules.
\item Let $$m_Q(q) =  {\tilde m} + n_0(\sum_id_i)q,$$
where $h_1, \ldots, h_{\mu}$ are generators of the ideal $I$ of degrees 
$d_1, \ldots, d_{\mu}\geq 1$ respectively,
 and $n_0\geq 1$ such that ${\bf m}^{n_0}\subseteq I$.
\item  We also denote $\dim_k \mbox{Coker}~\phi_{m,q}(\sQ)$ by
 $\coker~\phi_{m,q}(\sQ)$ (see the exact sequence (\ref{*}) above).
\item Let $a_1, \ldots, a_{\bar d}\in H^0(X,\sO_X(1))$ be such that we 
have a short
exact sequence of $\sO_X$-modules
$$0\rightarrow \sQ_i(-1)\by{a_i} \sQ_i \rightarrow \sQ_{i-1}\rightarrow 0,
\quad\mbox{for}\quad 0< i\leq {\bar d},$$
where $\sQ_{\bar d} = \sQ$ and $\sQ_i = \sQ/(a_{{\bar d}}, \ldots,a_{i+1})\sQ$,
for $0\leq i<{\bar d}$, with $\dim~\sQ_i = i$. (Such a sequence of $\{a_i\}_i$
exists, because $k$ is an infinite field, and since any coherent sheaf 
on $X$ has only finitely many associated points). We define
\begin{enumerate}
\item $C_0(\sQ) = h^0(X, \sQ)$, if ${\bar d} = 0$. If  ${\bar d} > 0$ then 
$$C_0(\sQ) = \mbox{min}\{\sum_{ i=0}^{\bar d}h^0(X,\sQ_i)\mid
 a_1,\ldots, a_{\bar d}~~~\mbox{is a}~~~\sQ-\mbox{sequence as above}\},$$

\item $$C_Q = (\mu)\left(h^0(X, \sQ({\tilde m}-1))
+\max\{\ell(Q_0), \ell(Q_1), \ldots, \ell(Q_{{\tilde m}-1})\}\right).$$
\item $$D_Q =  C_0(\sQ)\left[2{\bar d}({\tilde m}+
n_0(\sum_{1}^{\mu} d_i))\right]^{\bar d},$$
where $n_0$ , $\mu$, $d_i$ are given as in (3) above.

\item $$D_1(\sQ) = \mbox{max}\{h^1(X,\sQ), h^1(X,\sQ(1)), \ldots, 
h^1(X,\sQ({\tilde m}-1)) 
\},$$
\item $$D_0(\sQ) = h^0(X,\sQ({\tilde m})) +
2({\bar d} +1) \left(\mbox{max}\{q_0, |q_1|, \ldots, |q_{{\bar d}}|
\}\right),$$
where
$$\chi(X, \sQ(m)) = q_0{{m+{\bar d}}\choose{{\bar d}}}
+ q_1{{m+{\bar d}-1}\choose{{\bar d}-1}}+\cdots +
 q_{\bar d}$$
is the Hilbert polynomial of $\sQ$.
\end{enumerate}
\end{enumerate}
\end{notations}

The following lemma allows us to 
 reduce our various assertions about a graded module to 
assertions about cohomologies of 
the sheaf associated to the graded module.

\begin{lemma}\label{r1}
\begin{enumerate}
\item For  $m+q \geq m_Q(q)$, we have 
  $\coker~\phi_{m,q}(\sQ) = \ell(Q/I^{[q]}Q)_{m+q} = 0$.
\item For all $n\geq 0$ and $m\in \Z$ (where we define 
$Q_m = 0$, for $m<0$),
$$|\coker~\phi_{m,q}(\sQ) - \ell(Q/I^{[q]}Q)_{m+q} |\leq C_Q.$$
\end{enumerate}
\end{lemma}
\begin{proof}
For given $q=p^n$ and $m\geq 0 $, let 
 $\phi_{m, q}(Q):\oplus_iQ_{q-qd_i+m}\longrightarrow Q_{m+q}$ be the map such
 that  $Q_{q-qd_i+m}\rightarrow Q_{m+q}$ is given by 
 multiplication by
the element $h_i^q$. 
\vspace{5pt}

\noindent~(1)~~ To prove the first assertion note that  
$$m_Q(q) =  {\tilde m} + n_0(\sum_{i=1}^\mu d_iq) \geq 
{\tilde m}+d_iq \implies 
 q-qd_i+m\geq {\tilde m},$$
 for all $i$.
Hence the map $\phi_{m, q}(Q)$ is the map $\phi_{m, q}(\sQ)$ and 
therefore, $\coker~\phi_{m,q}(\sQ) = \ell(Q/I^{[q]}Q)_{m+q}$. Now, 
by the proof of Lemma~\ref{l1}, we have $\ell(Q/I^{[q]}Q)_{m+q} = 0$, as 
${m+q} \geq 
{\tilde m}+n_0\mu q$, since $\sum_id_i\geq\mu$.

\vspace{5pt}

\noindent~(2)\quad Note that
 $h^0(X, \sQ(t))\leq h^0(X, \sQ({\tilde m}-1))$, for all $t\leq  {\tilde m}-1$.
 
If $m+q < {\tilde m}$, then
$$|\coker~\phi_{m,q}(\sQ) - \ell(Q/I^{[q]}Q)_{m+q} |\leq h^0(X, \sQ(m+q))+
\ell(Q_{m+q})$$
$$ \leq h^0(X, \sQ({\tilde m}-1))+\mbox{max}\{\ell(Q_0),\ell(Q_1), 
\ldots, \ell(Q_{{\tilde m}-1})\}.$$
If $m+q \geq {\tilde m}$, then $h^0(X, \sQ(m+q))= \ell(Q_{m+q})$ and therefore 
$$|\coker~\phi_{m,q}(\sQ) - \ell(Q/I^{[q]}Q)_{m+q}| \leq 
|\sum_i\ell(\phi_{m,q}(Q)(Q_{q-qd_i+m}))-\ell(\phi_{m,q}
(\sQ)(H^0(X, \sQ(q-qd_i+m)))|.$$
Now, if $q-qd_i+m < 0$ then $Q_{q-qd_i+m} = 0$, and  $h^0(X, \sQ(q-qd_i+m)) \leq h^0(X, \sQ)$.
If  $q-qd_i+m \geq {\tilde m}$ then $Q_{q-qd_i+m} =  H^0(X, \sQ(q-qd_i+m))$.
This implies that 
 $$|\coker~\phi_{m,q}(\sQ) - \ell(Q/I^{[q]}Q)_{m+q} |\leq 
(\mu)\left(h^0(X, \sQ({\tilde m}-1))
+\max\{\ell(Q_0), \ell(Q_1), \ldots, \ell(Q_{{\tilde m}-1})\}\right).$$
Therefore $|\coker~\phi_{m,q}(\sQ) - \ell(Q/I^{[q]}Q)_{m+q}|
\leq C_Q$. This proves the second assertion.
\end{proof}

\begin{lemma}\label{l2} Let  $\sQ$ be  a coherent sheaf of $\sO_X$-modules of
 dimension ${\bar d}$.  Let $P$ and $P''$ be  locally-free sheaves  of $\sO_X$-modules
 which fit into a short exact sequence of locally-free
 sheaves of $\sO_X$-modules, where $b_i\geq 0$, 
\begin{equation}
\label{e15}0\longto P\longto  \oplus_{i} \sO_{X}(-b_i)
\longto P''\longto 0,~~~~\mbox{where}~~~ b_i 
\geq 0.\end{equation}
Then, for ${\tilde \mu} = \Rank(P)+ \Rank(P'')$, the following hold. 
\begin{enumerate}
\item
$$h^0(F^{n*}P\tensor \sQ(m)) \leq ({\tilde \mu}){C_{0}}(\sQ)(m^{{\bar d}}+1),
~~~~\mbox{for all}~~~ n, m \geq 0.$$
\item For each  $q = p^n$, let  ${m}_n\geq 0$ be an integer with the 
property  that, for all $i\geq 1$ and $m\geq {m}_n$, we have 
$h^i(X, F^{n*}P\tensor \sQ(m)) = 0$; then 
$$h^1(X, F^{n*}P\tensor \sQ(m)) \leq ({\tilde \mu})C_{0}(\sQ)(2{m}_n{\bar d})^{\bar d}, 
~~~\mbox{for all}~~~n, m \geq 0.$$ 
\item Moreover,
for all  $n, m\geq 0$, we have 
$$h^0(X, \sQ(m+q))\leq D_{0}(\sQ)(m+q)^{\bar d}
~~~\mbox{and}~~~h^1(X,\sQ(m))\leq D_{1}(\sQ).$$
\end{enumerate}
\end{lemma}
\begin{proof} Assertion~(3) is obvious from the definition of $D_1(\sQ)$ and 
$D_0(\sQ)$ given in Notations~\ref{n1}.

Let $\sQ_{{\bar d}} = \sQ$. Let
$a_{{\bar d}}, \ldots, a_1\in H^0(X, \sO_X(1))$ with the exact
sequence of $\sO_X$-modules
$$
0\longto {\sQ_i}(-1)\longby{a_i} \sQ_i\longto \sQ_{i-1}\longto 0,$$
where $\sQ_i = \sQ_{{\bar d}}/(a_{{\bar d}}, \ldots, a_{i+1})\sQ_{{\bar d}}$,
for $0 \leq  i \leq {\bar d}$, and realizing the minimal value $C_0(\sQ)$.
Now, by the exact sequence~(\ref{e15}),  we have the following
short exact
sequence of $\sO_X$-sheaves
$$0\longto F^{n*}P\tensor \sQ_i \longto
\oplus_j \sQ_i(-qb_j) \longto F^{n*}P''\tensor\sQ_i\longto 0.$$
This implies
$H^0(X, F^{n*}P\tensor\sQ_i) \into
 \oplus_jH^0(X, \sQ_i(-qb_j)).$
Therefore
\begin{equation}\label{e11}h^0(X, F^{n*}P\tensor\sQ_i)
\leq \sum_jh^0(X, \sQ_i(-qb_j)) \leq ({\tilde \mu})h^0(X, \sQ_i),\end{equation}
as $-b_j \leq  {0}$.
Since $F^{n*}P$ is a locally-free sheaf of $\sO_X$-modules, we have
\begin{equation}\label{e26} 0\longto F^{n*}P\tensor{\sQ_i}(m-1)
\longby{a_i} F^{n*}P\tensor\sQ_i(m)\longto
F^{n*}P\tensor\sQ_{i-1}(m)\longto 0,\end{equation}
which is a short exact sequence of $\sO_X$-sheaves.

\vspace{5pt}

\noindent{\bf Claim}\quad  For $m\geq 1$,
$$h^0(X, F^{n*}P\tensor \sQ_i(m))\leq ({\tilde \mu})\left[h^0(X, \sQ_{i})+\cdots+
h^0(X, \sQ_0)\right](m^{i}).$$

\noindent{\underline{Proof of the claim}}:\quad We prove the claim, by induction on $i$.
For $i =0$, the inequality holds as
$h^0(X, F^{n*}P\tensor \sQ_0(m))\leq ({\tilde \mu})h^0(X, \sQ_0)$
(as $\dim~\sQ_0 =0$).

Now, for $m\geq 1$, by the exact sequence~(\ref{e26}) and by induction on $i$, we have
$$h^0(X, F^{n*}P\tensor \sQ_i(m))  \leq h^0(X, F^{n*}P\tensor \sQ_i)
 + h^0(X, F^{n*}P\tensor \sQ_{i-1}(1))+ \cdots +
h^0(X, F^{n*}P\tensor \sQ_{i-1}(m))$$
$$ \quad\quad\quad\quad\quad\quad\quad\quad\leq 
({\tilde \mu})h^0(X, \sQ_i)+{\tilde \mu}
\left[h^0(X, \sQ_{i-1})+\cdots+
h^0(X, \sQ_0)\right](1+2^{i-1}+\cdots+m^{i-1})$$
$$ \leq ({\tilde \mu})
\left[h^0(X, \sQ_{i})+\cdots+
h^0(X, \sQ_0)\right]m^{i}.\quad\quad\quad\quad$$
This proves the claim.

This implies $h^0(X, F^{n*}P\tensor \sQ(m)) = h^0(X, F^{n*}P\tensor \sQ_{\bar d}(m))
\leq {\tilde \mu}C_0(\sQ)m^{\bar d},$
for all $m\geq 1$. 

Therefore, 
 and for all $m\geq 0$,
we have 
$$h^0(X, F^{n*}P\tensor \sQ(m))\leq
{\tilde \mu}C_0(\sQ)(m^{\bar d}+1).$$
This proves assertion~(1).

Let $h^j(X, F^{n*}P\tensor \sQ(m)) = 0$, for $m\geq {m}_n$, $j\geq 1$. If $m_n = 0$ 
then the assertion~(2) is obvious.  So we assume 
 $m_n \geq 1$. Then, by the exact sequence~(\ref{e26}) and descending
induction on $i$, we have  
$h^j(X, F^{n*}P\tensor \sQ_i(m)) = 0$, for all $m\geq {m}_n+{\bar d}$ and 
for $j\geq 1$.
Now, for $0\leq m < { m}_n+{\bar d}$,
$$
h^1(X, F^{n*}P\tensor \sQ_i(m))  \leq  h^0(X, F^{n*}P\tensor 
\sQ_{i-1}({ m}_n+{\bar d})) + \cdots + h^0(X, F^{n*}P\tensor 
\sQ_{i-1}(m+1))$$
$$\begin{array}{l} \leq ({\tilde \mu})\left[h^0(X, \sQ_{i-1})+\cdots + 
h^0(X, \sQ_{0})\right]   
      [({m}_n+{\bar d})^{i-1}+\cdots + 
((m+1))^{i-1}]\\
 \leq  ({\tilde \mu})\left[h^0(X, \sQ_{i-1})+\cdots + 
h^0(X, \sQ_{0})\right]   
({m}_n+{\bar d})^i,\end{array}$$
where the second inequality follows from the above claim. This implies, for all $0\leq 
m < m_n+{\bar d}$,
$$h^1(X, F^{n*}P\tensor \sQ(m))\leq 
({\tilde \mu})\left[h^0(X, \sQ_{d-1})+\cdots + 
h^0(X, \sQ_{0})\right]({m_n}+{\bar d})^{\bar d}
\leq ({\tilde \mu}){C_{0}(\sQ)}({m_n} + {\bar d})^{\bar d}.$$
Therefore 
$$h^1(X, F^{n*}P\tensor \sQ(m)) \leq 
({\tilde \mu}){C_{0}(\sQ)}(2{m_n}{\bar d})^{\bar d},$$ 
for all $m$, $n\geq 0$.
 This  completes 
the proof.\end{proof}

In the following lemma we write down a list of bounds on the cohomologies of 
the sheaves relevant to Theorem~\ref{t1}. 

\begin{lemma}\label{r2} Let $Q = \oplus_{m\geq 0}Q_m$ be a nonnegatively graded 
Noetherian $R$-module and let $\sQ$ be the coherent sheaf of $\sO_X$-modules 
associated to $Q$. Then 
\begin{enumerate}
\item $h^0(X, F^{n*}V\tensor \sQ(m)) \leq ({\mu}){C_{0}}(\sQ)(m^{{\bar d}}+1)$,
 for all $m$, $n\geq 0$.
\item $h^1(X, F^{n*}V\tensor \sQ(m)) \leq ({\mu}) (D_{Q})(q^{\bar d})$ and
$\sum_{1}^{\mu}h^1(X, \sQ(q-qd_i+m)) \leq ({\mu}) (D_{Q})(q^{\bar d})$, 
 for all $m$, $n\geq 0$.
\item $h^0(X, \sQ(m+q)) \leq D_0(\sQ)(m+q)^{\bar d}$, for all $m$, $n\geq 0$.
\item $h^1(X, \sQ(m)) \leq D_1(\sQ)$, for all $m\geq 0$.
\item $|\coker~\phi_{m,q}(\sQ) - \ell(Q/I^{[q]}Q)_{m+q} |\leq C_Q$, 
for all $n\geq 0$ and $m\in \Z$ (where we define 
$Q_m = 0$, for $m<0$).
\end{enumerate}
\end{lemma}
\begin{proof} Assertion~(1), (3) and (4) follow from Lemma~\ref{l2} 
and Assertion~(5) 
follows from Lemma~\ref{r1}~(2).

To prove Assertion~(2),
let $m_Q(q) = {\tilde m} + n_0(\sum_id_i)q$.
Note that, for $j\geq 1$ and $m+q \geq m_Q(q)$,
$$\sum_{i=1}^\mu H^j(X, F^{n*}\sO_X(1-d_i)\tensor\sQ(m)) = 
\sum_{i=1}^{\mu} H^j(X, \sQ(q-qd_i+m))= 0, $$
as $q-qd_i+m \geq {\tilde m}$. By Lemma~\ref{r1}, $\coker~\phi_{m,q}(\sQ) = 
0$. Therefore, by the  long exact sequence~(\ref{*}), 
$$m +q \geq {m_Q(q)} \implies h^j(X, F^{n*}V\tensor \sQ(m))= 0,~~~\mbox{for 
all}~~j\geq 1.$$
Hence, by Lemma~\ref{l2} (2), for all $m\geq 0$ and  $q = p^n$, 
where $P=\oplus_i\sO_X(1-d_i)$ or $V$, we have
\begin{equation}\label{e4}h^1(X, F^{n*}P\tensor \sQ(m))\leq 
(\mu)C_{0}(\sQ)(2m_Q(q){\bar d})^{\bar d} \leq \mu D_Qq^{\bar d},\end{equation}
where ${\bar d}$ is the dimension of the support of $\sQ$ and 
\begin{equation}\label{e21} D_Q = 
C_0(\sQ)(2{\bar d})^{\bar d}({\tilde m}+
n_0(\sum d_i))^{\bar d} = C_0(\sQ)\left[2{\bar d}({\tilde m}+
n_0(\sum d_i))\right]^{\bar d}. 
\end{equation}
This proves Assertion~(2) and hence the lemma.
\end{proof}

\vspace{10pt}

\begin{lemma}\label{l22}Let Let $X=\rm{Proj}~R$ be a projective $k$-scheme of 
dimension $d-1$ with a very ample invertible sheaf $\sO_X(1)$. 
Let 
$$0\longto \sQ'\longto \sM'\longby{f}
\sM'' \longto \sQ''\longto 0,$$
be an exact sequence of sheaves of coherent $\sO_X$-modules such that 
$\sQ'$ and $\sQ''$ are coherent sheaves of $\sO_X$-modules with support of  
dimensions $< d-1$. Then, for all $m$, $n \geq 0$,  
\begin{enumerate}
\item $$|\coker~\phi_{m,q}(\sM')-\coker~\phi_{m,q}(\sM'')|\leq C(f)(m+q)^{d-2},
$$ where 
$$C(f) = {\mu}\left[2C_0(\sQ'')+D_0{\sQ''}+2C_0(Q')+D_0(Q')+2D_{\sQ'}+D_1(\sQ')
\right].$$
Moreover 
\item if
$M'$ and $M''$ are two
nonnegatively  graded $R$-modules associated to $\sM'$ and $\sM''$ 
respectively then 
$$|\ell(M'/I^{[q]}M')_{m+q}-\ell(M''/I^{[q]}M'')_{m+q}| \leq C(f)(m+q)^{d-2}
+C_{M'}+ C_{M''}.$$
\end{enumerate}
\end{lemma}
\begin{proof} The above  exact sequence we can break into following two short 
exact sequence of $\sO_X$-sheaves
$$0\longto \sQ'\longto \sM'\longto K\longto 0,$$
$$0\longto K \longto \sM''\longto \sQ''\longto 0.$$
 For a locally-free sheaf $P$ of $\sO_X$-modules,
both the above short exact sequences remain exact after tensoring with
 $(F^{n*}P)(m)$, for all $m\geq 0$ and $n\geq 0$.
Therefore we have long exact sequence of cohomologies
$$0 \longto H^0(X, F^{n*}P\tensor \sQ'(m))\longto
H^0(X, F^{n*}P\tensor\sM') $$
$$ \longto H^0(X, F^{n*}P\tensor K(m))\longto H^1(X, F^{n*}P
\tensor \sQ'(m))\longto \cdots $$
and
\begin{equation}\label{e***}0\longto H^0(X, F^{n*}P\tensor K(m))
\longto H^0(X, F^{n*}P\tensor \sM''(m))
 \longto H^0(X, F^{n*}P\tensor Q''(m)) \end{equation}

For a coherent sheaf $L$ of $\sO_X$-modules
$$\begin{array}{lcl}
\coker~\phi_{m,q}(L) & = & h^0(X, F^{n*}\sO_X(1)\tensor L(m))\\
 & &  -
\sum_{i=1}^sh^0(X, F^{n*}\sO_X(1-d_i)\tensor L(m))
+h^0(X, F^{n*}V\tensor L(m)).\end{array}$$

By (\ref{e***}),  $$|h^0(X, F^{n*}P\tensor K(m))-
h^0(X, F^{n*}P\tensor \sM''(m))|\leq
h^0(X, F^{n*}P\tensor Q''(m)).$$
Therefore, by Lemma~\ref{r2}, we have
$$|\coker~\phi_{m,q}(K)-\coker~\phi_{m,q}(\sM'')|\leq $$
$$ h^0(X, \sQ''(m+q))+ \sum_{i=1}^sh^0(X, \sQ''(m+q-qd_i))
+h^0(X, F^{n*}V\tensor \sQ''(m))$$
$$\leq  D_0(Q'')(m+q)^{{d-2}} 
+ \mu C_0(Q'')(m^{d-2}+1) + \mu C_0(Q'')(m^{d-2}+1),$$
as $h^0(X, \sQ''(m+q-qd_i)) \leq h^0(X, \sQ''(m+q))$.

Therefore 
\begin{equation}\label{e25}|\coker~\phi_{m,q}(K)-\coker~\phi_{m,q}(\sM'')|
\leq \mu \left[2C_0(Q'')+D_0(Q'')\right](m+q)^{d-2}.\end{equation}
Similarly, since for a locally free sheaf $P$ and for $m, n\geq 0$, we have 
$$|h^0(X, F^{n*}P\tensor \sM'(m))-h^0(X, F^{n*}P\tensor K(m))|\leq 
h^0(X, F^{n*}P\tensor \sQ'(m))+h^1(X, F^{n*}P\tensor \sQ'(m)),$$
we deduce,
 by Lemma~\ref{r2} 
\begin{equation}\label{e23}
|\coker~\phi_{m, q}(\sM')-\coker~\phi_{m,q}(K)|\end{equation}
$$\leq \mu \left[2C_0(\sQ')+D_0(\sQ')\right](m+q)^{d-2}+
2 \mu  D_{Q'}q^{d-2}+D_1(\sQ').$$
Therefore, by (\ref{e25}) and (\ref{e23}), for all $m$, $n\geq 0$, we have
\begin{equation}\label{ee22}|\coker~\phi_{m, q}(\sM')-\coker~\phi_{m, q}
(\sM'')|\end{equation}
$$\leq 
{\mu}\left[2C_0(\sQ'')+D_0{\sQ''}+2C_0(Q')+D_0(Q')+2D_{\sQ'}+D_1(\sQ')
\right](m+q)^{d-2} = C(f)(m+q)^{d-2}.$$

 Now Assertion~(2) follows from Lemma~\ref{r2}~(5).  
\end{proof}

\begin{lemma}\label{l4} Let $Y=X_{red}$, which is  a reduced projective $k$-scheme 
of dimension 
$d-1$ with a very ample invertible sheaf $\sO_Y(1)$. 
 Then, for a coherent sheaf $\sN$ of $\sO_Y$-modules, there exist an
 integer $m_2 \geq  1$, depending on 
$\sN$, such that we have an exact sequence of sheaves of
 $\sO_Y$-modules 
$$0\longto Q'\longto \oplus^{p^{d-1}}\sN(-m_2)\longto
 F_*\sN\longto Q''\longto 0,$$
where $Q'$ and $Q''$ are coherent sheaves of $\sO_Y$-modules with support of  
dimensions $< d-1$.
\end{lemma}
\begin{proof} Let $x_1, \ldots, x_{s_1}$ be the generic points of the maximal
 components $Y_1, \ldots, Y_{s_1}$ of
 $Y$, where $\dim~Y_i = \dim~Y$. We choose $f\in H^0(Y,\sO_Y(1))$ 
such that $f$ does not vanish on $\sO_{Y,x_i}$, for all $i$. Note that 
$\sO_{Y,x_i}$ is the function field of $Y_i$.
In particular $x_1, \ldots, x_{s_1} \in D_+(f)$ where $D_+(f)$
 is a reduced affine variety. Let us denote $D_+(f)$ by $U_f$.
Let $\Gamma(U_f, \sO_Y) = A$. Let $p_1, \ldots, p_{s_1}\in Spec~A$ be the  
prime ideals corresponding to the points $x_1,\ldots, x_{s_1}$ and let 
$S = A\setminus p_1\cup \cdots\cup p_{s_1}$. Then, by Chinese Remainder theorem 
$$ S^{-1}A \simeq \sO_{Y,x_1}\times \cdots \times \sO_{Y,x_{s_1}}~~~
\mbox{ and}~~~  S^{-1}\Gamma(U_f, \sN) \simeq \sN_{x_1}\times \cdots \times 
\sN_{x_{s_1}},$$
and 
$$S^{-1}\Gamma(U_f, F_*\sN) \simeq (F_*\sN)_{x_1}\times \cdots \times 
(F_*\sN)_{x_{s_1}} =  F_*(\sN_{x_1})\times \cdots \times 
F_*(\sN_{x_{s_1}}).$$

Now if $\sN_{x_i}$ is of rank $m_i$  as $A_{x_i}$-module then 
$F_*\sN_{x_i}$ is of rank $p^{d-1}m_i$ as $A_{x_i}$-module, as $F_*A$ is of
rank $p^{d-1}$ over $A$ and $F_*\sN_{x_i}$ is of rank $m_i$ over $F_*A_{x_i}$.
 This implies that there is a $\sO_{Y, x_i}$-linear isomorphism 
$\phi_i:\oplus^{p^{d-1}}\sN_{x_i} \longrightarrow (F_*\sN)_{x_i}$, which 
 gives an $S^{-1}A$-linear isomorphism 
$$\phi: \oplus^{p^{d-1}}S^{-1}\Gamma(U_f,\sN) \longrightarrow 
S^{-1}\Gamma(U_f, F_*\sN).$$
Since $\sN$ is a coherent $\sO_Y$-sheaf, one  
can choose ${\tilde s}\in S$ and 
${\tilde \phi}:\oplus^{p^{d-1}}\Gamma(U_f,\sN)\longrightarrow  \Gamma(U_f,
 F_*\sN)$  such that 
${\tilde \phi}$ maps to ${\tilde s}\cdot\phi$ under the localization map
$$\displaystyle{\Hom_A\left(\Gamma(U_f,\oplus^{p^{d-1}}\sN), \Gamma(U_f,
 F_*\sN)\right) \longto 
\Hom_{S^{-1}A}\left(S^{-1}\Gamma(U_f,\oplus^{p^{d-1}}\sN), 
S^{-1}\Gamma(U_f, F_*\sN)\right))}$$

Therefore there exists $n\geq 1$ and  
${\psi} \in 
\Gamma(Y, {\mathcal Hom}_{\sO_Y}(\oplus^{p^{d-1}}\sN, F_*\sN)\tensor 
\sO_Y(n))$ such that $\psi$ restricts to $f^n\cdot {\tilde s}\cdot\phi$
on the open set $U_f$
  (see [Ha], Lemma~5.14). 
This gives an exact sequence of $\sO_Y$-linear maps
$$0\longto \mbox{Ker}~{\psi}\longto  \oplus^{p^{d-1}}\sN(-n) 
\longby{\psi} F_*\sN\longto  \mbox{Coker}~{\psi}\longto 0.$$
Since $\psi$ localizes to a unit multiple of $\phi$, 
it is an isomorphism at the points
$x_1, x_2,\ldots, x_{s_1}$, which implies that   
the dimensions of the support of $\mbox{Ker}~{\psi}$ and
 $\mbox{Coker}~{\psi}$ are 
 $< \dim~Y$. This proves the lemma.
\end{proof}

\begin{lemma}\label{l6}Let $M$ be a nonnegatively graded finitely 
generated $R$-module and let $\sM $ be the associated  coherent sheaf of $\sO_X$-modules.
 Then there exists 
a nonnegative integer $s$ ({\it e.g.}, 
$s\geq 0$ such that $(\mbox{nilradical}~R)^{p^{s}} = 0)$
 and an integer 
$m_2\geq 1$ (depending on $\sM$ and $q' = p^s$) such that  
\begin{enumerate}
\item there is a long exact sequence of sheaves of $\sO_X$-modules 
$$0\longto Q'\longto \oplus^{p^{d-1}}(F_*^s\sM)(-m_2)\longby{g}
 F_*^{s+1}\sM\longto Q''\longto 0,$$
where $Q'$ and $Q''$ are coherent sheaves of $\sO_X$-modules with support of  
dimensions $< d-1$.
\item There is a constant $C(g)$ (as given in Lemma~\ref{l22}~(1)) such that, for all 
$m$, $n\geq 0$,
$$ |p^{d-1}\ell(M/I^{[qq']}M)_{(m+q-m_2)q'}-
\ell(M/I^{[qq'p]}M)_{(m+q)q'p}|\leq C(g)(m+q)^{d-2}+2C_{M}.$$
\end{enumerate}
\end{lemma}

\begin{proof}
Let $p^{s}$ be an integer such that 
$(\mbox{nilradical}~R)^{p^{s}} = 0$. Then $\sN = F^s_*{\sM}$ is a coherent 
$\sO_X$-modules annhilated by the nilradical of $\sO_X$.
Consider the canonical short exact sequence of $\sO_X$-modules obtained from
 Equation~(\ref{e2}),
\begin{equation}\label{e5}  0\longto F^{n*}V\tensor\sN(m)\longto
 \oplus_i \sN(q-qd_i+m))\longto  
\sN(q+m)\longto 0.\end{equation}

Since $\sN$ is annihilated by the nilradical of $\sO_X$,  the action 
of $\sO_X$ on $\sN$ filters through a canonical action of 
 $\sO_{X_{red}}$ on $\sN$. 

 Since $\sN$ is also a sheaf of $\sO_{X_{red}}$-modules, 
by Lemma~\ref{l4}, there exists constant 
 $m_2$ depending on $\sN$ and $\sO_{X_{red}}$ such that we have a short exact
 sequence of $\sO_{X_{red}}$-modules and hence of $\sO_X$-modules,
$$0\longto Q'\longto \oplus^{p^{d-1}}\sN(-m_2)\longby{g}
 F_*\sN\longto Q''\longto 0,$$
where $Q'$ and $Q''$ are coherent sheaves of $\sO_{X_{red}}$-modules (and hence 
coherent sheaves of $\sO_X$-modules) with support of  
dimensions $\leq d-2$.
Therefore, by Lemma~\ref{l22}~(1), there is a constant $C(g)$ for the map $g$ such that 
$$|\coker~\phi_{m, q}(\oplus^{p^{d-1}}\sN(-m_2)) - 
\coker~\phi_{m,q}(F_*\sN)|\leq C(g)(m+q)^{d-2},$$
 for all $m, n\geq 0$
Therefore 
\begin{equation}\label{e30}
|p^{d-1}\coker~\phi_{m-m_2, q}(\sN) - \coker~\phi_{m,q}(F_*\sN)|\leq C(g)(m+q)^{d-2}.
\end{equation}

We note that, for any locally-free sheaf $P$ of $\sO_X$-modules,
 using the projection formula, we have (since $k$ is perfect) 
$$  h^i(X, F^{(n+1+s)*}P\tensor \sM(mpq')) = 
h^i(X, F^{s*}\left(F^{n+1*}P\tensor \sO(mp)\right)\tensor \sM)
\quad\quad\quad\quad\quad $$
$$\quad\quad\quad\quad\quad = h^i(X, \left(F^{n+1*}P\tensor \sO(mp)\right)\tensor \sN)
=   h^i(X, F^{n*}P\tensor\sO(m)\tensor F_*\sN) $$
Therefore 
\begin{equation}\label{e13}
\coker~\phi_{(mp)q', qq'p}(\sM) = \coker\phi_{m, q}(F_*\sN),
\end{equation}
Similarly 
$$ h^i(X, F^{(n+s)*}P\tensor \sM((m-m_2)q')) =
h^i(X, F^{n*}P \tensor\sO(m-m_2)\tensor F_*^s\sM) = 
h^i(X, F^{n*}P \tensor \sN(m-m_2)).
$$
Therefore 
\begin{equation}\label{e16}\coker~\phi_{(m-m_2)q', qq'}(\sM) = 
\coker\phi_{m-m_2, q}(\sN).
\end{equation}

Hence, by (\ref{e30}), 
$$|p^{d-1}\coker~\phi_{(m-m_2)q', qq'}(\sM) - \coker~\phi_{(mp)q',qq'p}(\sM)|
\leq C(g)(m+q)^{d-2}.$$
Therefore, by Lemma~\ref{r2}~(5), 
$$ |p^{d-1}\ell(M/I^{[qq']}M)_{(m+q-m_2)q'}-
\ell(M/I^{[qq'p]}M)_{(m+q)q'p}|\leq C(g)(m+q)^{d-2}+2C_{M},$$
for all $m$, $n\geq 0$. 
\end{proof}

\begin{defn}\label{d1}For a pair $(M, I)$, where 
 $M$ is a finitely generated nonnegatively  graded $R$-module and $I$ is a 
homogeneous ideal of $R$ such that $\ell(R/I)<\infty$. We define sequences of 
functions $\{f_n:\R\longto \R\}_{n\in \N}$ and $\{g_n:\R\longto \R\}_{n\in \N}$ 
as follows:

For $n\in \N$, let $q = p^n$. Define
$$f_n(x) = g_n(x) = 0,~~~\mbox{if}~~~x<0.$$

Let $x\geq 0$ then  
${{m}}/{q} \leq x < {m+1}/{q}$, for some integer $m\geq 0$. We define   
$$\begin{array}{lcl}
f_n(x) &  = &  \displaystyle{\frac{1}{q^{d-1}}
\ell\left({M}_m\right),\quad\mbox{if}\quad 0\leq x < 1}\\
& = & \displaystyle{\frac{1}{q^{d-1}}
\ell\left(\frac{M}{I^{[q]}M}\right)_{m},\quad\mbox{if}\quad
1\leq \frac{{m}}{q} \leq x < \frac{m+1}{q}}.\end{array}$$
$$\begin{array}{lcl}
g_n(x) & = & \displaystyle{f_n(x)\quad\mbox{if}~\quad x= \frac{m}{q}},\\
&  = & \displaystyle{(1-t)f_n\left(\frac{m}{q}\right)+tf_n\left(\frac{m+1}{q}\right),
~~~\mbox{if}~~~
x= (1-t)\left(\frac{m}{q}\right)+t\left(\frac{m}{q}\right)~~~\mbox{where}~~~t\in [0, 1)}.\end{array}
$$
\end{defn}
\vspace{10pt}

\begin{prop}\label{t2} For a given pair $(M, I)$ as in Definition~\ref{d1} above, and 
where $\dim~R= d\geq 2$, 
the sequence $\{f_n:\R\longto \R\}_{n\in \N}$ 
is a uniformly convergent 
sequence  of compactly supported functions.

More precisely, there exists $n_0\in \N$  and a constant $C$ depending on $M$, 
such that 
\begin{equation}\label{**}|f_n(x)-f_{n_1}(x)|\leq C/p^n,~~~\mbox{ for all}~~~~
 n_1 \geq n\geq n_0~~~\mbox{and for all}~~~x\in \R.\end{equation}
\end{prop}
\begin{proof}
Note that $\dim~R \geq 2$. Therefore, for $X = \mbox{Proj}~R$, 
we have  $\dim~X \geq 1$. Let $\sM$ be the coherent sheaf of
 $\sO_X$-modules associated to $M$.

\vspace{10pt} 

\noindent{(A)~~ Let  $x<1 $}.

 If $x < 0$ then
 $f_n(x) = f_{n+1}(x) = 0$, for all $n \geq 1$.

 Let $ 0 \leq x < 1$. Then $m/{q}\leq x < (m+1)/q$, for some integer
 $ 0 \leq m < q$. Hence 
$$\frac{mp+n_1}{qp} \leq 
x < \frac{mp+n_1+1}{qp},~~~\mbox{for some integer}~~~~  0\leq n_1 < p,
~~ \mbox{with}~~~ 
mp+n_1 < qp.$$
Therefore, 
$f_n(x) = (1/q^{d-1})\ell(M_m)$ and $f_{n+1}(x) = 
(1/(qp)^{d-1})\ell(M_{mp+n_1})$.

If $m\leq {\tilde m}$ (${\tilde m}$ is defined for $M$ as in
 Notations~\ref{n1}) then 
$$|f_n(x) - f_{n+1}(x)| < \left|\frac{(\ell(M_0)+\cdots + 
\ell(M_{\tilde m})}{q^{d-1}} + \frac{(\ell(M_0)+\cdots + 
\ell(M_{{\tilde mp}+n_1})}{(qp)^{d-1}}\right| \leq 
\frac{2\sum_{0}^{{\tilde m}p+(p-1)}\ell(M_i)}{q^{d-1}}.$$
If $ q > m > {\tilde m}$ then (using Hilbert polynomials) 
$$\ell(M_m) = {\tilde e_0}{m}^{d-1} + 
{\tilde e_1}m^{d-2} + \cdots + {\tilde e}_{d-1}$$
$$\ell(M_{mp+n_1}) = {\tilde e_0}(mp+n_1)^{d-1} + 
{\tilde e_1}(mp+n_1)^{d-2} + \cdots + {\tilde e}_{d-1},$$
 for some rational  numbers 
 ${\tilde e_0}, \cdots, {\tilde e_{d-1}}$ which are invariant of $(\sM, \sO_X(1))$. 
In this case 
$$|f_n(x)-f_{n+1}(x)| \leq  \frac{(d-1){\tilde e_0}+|{\tilde e_1}|+\cdots +
|{\tilde e_{d-1}}|}{q}.$$

This implies that, for 
${\tilde C_2}(M) =  2\sum_{0}^{{\tilde m}p+(p-1)}\ell(M_i) + 
(d-1){\tilde e_0}+|{\tilde e_1}|+\cdots +
|{\tilde e_{d-1}}|$,
\begin{equation}\label{e27}
|f_n(x)-f_{n+1}(x)| \leq  \frac{{\tilde C_2}(M)}{q},
\quad\mbox{for all}\quad x <1\quad\mbox{for all}\quad n\geq 0.\end{equation}
\vspace{10pt} 

\noindent~~(B)~~Let $ x\geq 1$. 
We fix two integers $m_2$ and $q' = p^s$ (as in Lemma~\ref{l6})
such that we have  an exact sequence of sheaves of $\sO_X$-modules,
$$0\longto Q'\longto \oplus^{p^{d-1}}(F_*^s\sM)(-m_2)\longby{g}
 F_*^{s+1}\sM\longto Q''\longto 0.$$

Let $s\in R_1$ which avoids all minimal primes of the ring
$R$ (note that $R$ is a standard graded ring and $k$ is infinite). 
For $0\leq n_1 <q'$ and $0\leq n_2 <p$, we consider the  following exact sequences of 
graded $R$-modules
$$0\longto Q'_{n_1}\longto M(-m_2q')\longby{f_{n_1}}
 M(n_1)\longto  Q''_{n_1}\longto 0,$$
where $f_{n_1}$ is the multiplication  map given by 
$s^{n_1+m_2q'}$. 
This induces canonical 
exact sequences of 
sheaves of $\sO_X$-modules
$$0\longto \sQ'_{n_1}\longto \sM(-m_2q')\longby{f_{n_1}}
 \sM(n_1)\longto \sQ''_{n_1}\longto 0,$$
Similarly we have exact sequences of graded $R$-modules
$$0\longto K'_{n_2, n_1}\longto M\longby{h_{n_2,n_1}}
 M(n_1p+n_1)\longto K''_{n_2, n_1}\longto 0$$
where $h_{n_2, n_1}$  is the multiplication  map given by 
$s^{n_1p+n_2}$.
This induces exact sequences of 
sheaves of $\sO_X$-modules 
$$0\longto \sK'_{n_2, n_1}\longto \sM\longby{h_{n_2,n_1}}
 \sM(n_1p+n_1)\longto \sK''_{n_2, n_1}\longto 0.$$
By construction,  each of the sheaves $\sQ'$, $\sQ''$, 
$\sQ'_{n_1}$, $\sQ''_{n_1}$, $\sK'_{n_2,n_1}$ and $\sK''_{n_2,n_1}$, 
has support  of  dimension $<d-1$.

Let 
$${\tilde C_0}(M) = \max_{0\leq n_1 <q',~~~0\leq n_2 < p}
\left\{C(f_{n_1})+C_{Q'_{n_1}}+ C_{Q''_{n_1}}, C(g)+ 2C_M,
C(h_{n_2, n_1})+ C_{K'_{n_2, n_1}}+C_{K''_{n_2, n_1}}
\right\},$$ 
where $C(f_{n_1})$, $C(g)$ and $C(h_{n_2, n_1})$ are 
the constants (see Lemma~\ref{l22})  
associated to the maps $f_{n_1}$, $g$ and $h_{n_2, n_1}$ respectively.

Since $x\geq 1$, for given $q = p^n$, 
 there exists a unique integer $m\geq 0$, such that 
$ (m+ q)/q \leq x < (m+q+1)/q $. Therefore,  
for  $q'=p^s$  we have 
$$\frac{(m+q)q'+n_1}{qq'} \leq x <\frac{(m+q)q'+n_1+1}{qq'},
~~~~~~~\mbox{for some}~~~n_1 < q'$$
and 
$$\frac{(m+q)q'p+n_1p+n_2}{qq'p}
 \leq x < \frac{(m+q)q'p+n_1p+n_2+1}{qq'p},~~~~~~\mbox{for some}~~~ n_2 < p.$$
Hence, by definition 
$$f_{n+s}(x) = \frac{1}{(qq')^{d-1}}\ell\left(\frac{M}{I^{[qq']}M}\right)_{(m+q)q'+n_1}~~~\mbox{and}$$
$$f_{n+s+1}(x) = \frac{1}{(qq'p)^{d-1}}\ell\left(\frac{M}{I^{[qq'p]}M}\right)_{(m+q)q'p+n_1p+n_2}.$$

Let $m_{M}(q) = {\tilde m}+ n_0(\sum_id_i)q$ (defined as in Notations~\ref{n1}). If 
$m\geq m_{M}(q)$  then 
 we have  $mq'\geq m_{M}(qq')$ and 
$mq'p\geq m_{M}(qq'p)$. Therefore, by Lemma~\ref{r1}~(1), for $m\geq m_{M}(q)$, 
$$\ell\left(\frac{M}{I^{[qq']}M}\right)_{(m+q)q'
+n_1} = \ell\left(\frac{M}{I^{[qq'p]}M}\right)_{(m+q)q'p+n_1p+n_2} = 0,$$
which implies 
$|f_{n+s}(x)-f_{n+s+1}(x)| = 0$.

Therefore we can assume 
$m\leq m_M(q)$ and hence can assume that 
$(m+q)^{d-2} \leq L_0q^{d-2}$, where
 $L_0 = ({\tilde m}+ n_0(\sum_id_i)+1)^{d-2}$.

We have 
$$|f_{n+s}(x)-f_{n+s+1}(x)| = |f_{n+s}(\frac{(m+q)q'+n_1}{qq'})-
f_{n+s+1}(\frac{(m+q)q'p+n_1p+n_2}{qq'p})|.$$

Hence we have 
$$ |f_{n+s}(x)-f_{n+s+1}(x)|\leq  A_1(x)+ A_2(x) + A_3(x),$$
 where 
$$A_1(x) = |f_{n+s}(\frac{(m+q)q'+n_1}{qq'})-f_{n+s}(\frac{(m+q-m_2)q'}{qq'})|$$
$$A_2(x) = |f_{n+s}(\frac{(m+q-m_2)q'}{qq'}) - f_{n+s+1}(\frac{(m+q)q'p}{qq'p})|$$
$$A_3(x) = |f_{n+s+1}(\frac{(m+q)q'p}{qq'p}) -f_{n+s+1}(\frac{(m+q)q'p+n_1p+n_2}{qq'p}).$$

Now 
$$A_1(x) = \frac{1}{(qq')^{d-1}}\left|\ell\left(\frac{M}{I^{[qq']}M}\right)_{(m+q)q'+n_1}-
\ell\left(\frac{M}{I^{[qq']}M}\right)_{(m-m_2+q)q'}\right|$$
$$A_1(x) \leq 
\frac{C(f_{n_1})(m+q)^{d-2}+C(Q'_{n_1})+C(Q''_{n_1})}{(qq')^{d-1}}\leq 
\frac{1}{qq'}\frac{{\tilde C_0(M)} L_0}{(q')^{d-2}}.$$

$$A_2(x) = \frac{1}{(qq'p)^{d-1}}\left|p^{d-1}\ell\left(
\frac{M}{I^{[qq']}M}\right)_{(m+q-m_2)q'}-
\ell\left(\frac{M}{I^{[qq'p]}M}\right)_{(m+q)q'p}\right|$$
$$A_2(x)\leq \frac{C(g)(m+q)^{d-2}+2C_M}{(qq'p)^{d-1}} \leq 
\frac{{\tilde C_0(M)}L_0q^{d-2}}{(qq'p)^{d-1}}\leq
\frac{1}{qq'}\frac{{\tilde C_0(M)}L_0}{q'^{d-2}p^{d-1}}.$$

$$A_3(x) = \frac{1}{(qq'p)^{d-1}}\left|\ell\left(\frac{M}{I^{[qq'p]}M}
\right)_{(m+q)q'p}-
\ell\left(\frac{M}{I^{[qq'p]}M}\right)_{(m+q)q'p+n_1p+n_2}\right|$$
$$A_3(x)\leq \frac{C(h_{n_2,n_1})(mq'p+qq'p)^{d-2}+C_{K'_{n_2,n_1}}+ 
C_{K''_{n_2,n_1}}}{(qq'p)^{d-1}} 
\leq \frac{1}{qq'}\frac{{\tilde C_0(M)} L_0}{p}.$$

Therefore 
$$|f_{n+s}(x)-f_{n+s+1}(x)|\leq A_1(x)+A_2(x)+A_3(x)
\leq \frac{L_0}{qq'}\left[\frac{{\tilde C_0(M)}}{(q')^{d-2}}+
 \frac{{\tilde C_0(M)}}{q'^{d-2}p^{d-1}}
+\frac{{\tilde C_0(M)}}{p}\right].$$
Let
${\tilde C_1}(M) = 3L_0{\tilde C_0(M)}$.
In particular ${\tilde C_1}(M)$
is a constant (which depends only on $M$) such that 
\begin{equation}\label{e28} |f_{n+s}(x)-f_{n+s+1}(x)|\leq 
{\tilde C_1}(M)/(qq') = 
{\tilde C_1}(M)/p^{n+s},~~~\mbox{for all}~~~n\geq 0~~~
\mbox{and}~~~x\geq 1.
\end{equation}
Since 
$$|{\tilde C_1}(M)/p^{n_0} + {\tilde C_1}(M)/p^{n_0+1}+\cdots 
|\leq 2{\tilde C_1}(M)/p^{n_0},$$
for $C \geq 2{\tilde C_1}(M)$
we get 
 $$|f_n(x)-f_{n_1}(x)|\leq C/p^n,~~~\mbox{ for all}~~~~
 n_1 \geq n\geq n_0~~~\mbox{and for all}~~~x\geq 1.$$
Combining this with (\ref{e27}), we get that for any
 $C\geq 2{\tilde C_2}(M)+2{\tilde C_1}(M)$ 
 $$|f_n(x)-f_{n_1}(x)|\leq C/p^n,~~~\mbox{ for all}~~~~
 n_1 \geq n\geq n_0~~~\mbox{and for all}~~~x\in \R.$$
This proves the proposition.\end{proof}

\vspace{10pt}

\noindent{\bf Proof of Theorem~\ref{t1}}:\quad By  Remark~\ref{r11}, we may assume $k 
= {\bar k}$. For $n\in \N$, let $f_n:\R\longto \R$ and $g_n:\R\longto \R$ be
 functions as given in Definition~\ref{d1}.

\vspace{5pt}

\noindent{\bf Claim}\quad Both the sequences $\{f_n\}_n$ and $\{g_n\}_n$
converge uniformly and to the same limit function.

\vspace{5pt}

\noindent{\underline{Proof of the claim}:
Let $ q =p^n$ and $x\in \R$. If $x < 0$ then ${f_n}(x) = g_n(x) = 0$, for all
$n\geq 0$. 

Let $x \geq 0$ then 
  $x = (1-t)\frac{\lfloor xq\rfloor}{q}+t\frac{\lfloor xq\rfloor+1}{q} $,
 for some $t \in [0, 1)$.
Therefore 
$$f_n(x) = \frac{1}{q^{d-1}}\ell\left(\frac{M}{I^{[q]}M}
\right)_{\lfloor xq\rfloor}~~~\mbox{and}~~~
 g_n(x) = \frac{(1-t)}{q^{d-1}}\ell\left(\frac{M}{I^{[q]}M}\right)_{\lfloor xq\rfloor}+
\frac{t}{q^{d-1}} \ell\left(\frac{M}{I^{[q]}M}\right)_{\lfloor xq\rfloor+1}.$$

Let $$0\longto Q'\longto M(-1)\longby{f}M\longto Q''\longto 0,$$
be the exact sequence of graded $R$-modules where the map $f$ is given 
by multiplication by an element $s\in R_1$, By choosing such an $s$ which 
avoids all minimal primes of $M$, we ensure that support of each of $Q'$ and $Q''$ 
is of dimension $<d$.
If $$0\longto \sQ'\longto \sM(-1)\longby{f}\sM\longto \sQ''\longto 0,$$
is the associated  exact sequence of sheaves of $\sO_X$-modules then 
by Lemma~\ref{l6}~(2) and Lemma~\ref{r1}~(1),
$$|\ell\left(\frac{M}{I^{[q]}M}\right)_{\lfloor xq\rfloor}-
\ell\left(\frac{M}{I^{[q]}M}\right)_{\lfloor xq\rfloor+1}|\leq (C(g)+2C_M)
L_0q^{d-2} = {C_1}q^{d-2},$$ 
for all $n\geq 1$ and $x\geq 0$,
where  
$L_0= ({\tilde m}+n_0(\sum_i d_i)+1)^{d-2}$.

This implies, for all $n\geq 1$ and $x \geq 0$, we have,
$$ |f_n(x) - g_{n}(x)|
= 
\frac{t}{q^{d-1}}\left|\ell\left(\frac{M}{I^{[q]}M}\right)_{\lfloor xq\rfloor} - 
\ell\left(\frac{M}{I^{[q]}M}\right)_{\lfloor xq\rfloor+1}\right|
\leq  \frac{C_1}{p^n}. $$
By  Proposition~\ref{t2}, there is a constant $C$ depending on  $M$ and 
$n_0\in \N$ such that 
$$|f_n(x)-f_{n_1}(x)|\leq C/p^n,~~~\mbox{ for all}~~~~
 n_1 \geq n\geq n_0~~~\mbox{and for all}~~~x\in \R.$$
This implies,   
$$|g_n(x)-g_{n_1}(x)|\leq |g_n(x)-f_{n}(x)|+|f_n(x)-f_{n_1}(x)|+|f_{n_1}(x)-g_{n_1}(x)|
\leq \frac{C_1}{p^n}+\frac{C}{p^n}+\frac{C_1}{p^{n_1}}.$$
Therefore we have 
$$|f_n(x)-f_{n_1}(x)|, |f_n(x)-g_{n}(x)|, |g_n(x)-g_{n_1}(x)|\leq 
\frac{2C_1+C}{p^n},$$
for all $n_1\geq n\geq n_0$ and for all $x\in \R$.

Hence $\{f_n\}_n$ and $\{g_n\}_n$ are
 uniformly 
convergent sequences with the same limit. This proves the claim.

\vspace{5pt}
Let $f:\R\longto \R$ be the limit function given by
$$f(x) = \lim_{n\to \infty}{f_n}(x)dx = \lim_{n\to \infty}{g_n}(x)dx.$$

By the proof of Lemma~\ref{l1}, $g_n$ is a continous function with 
the support~$g_n \subseteq [0, (n_0\mu) +l/q]$. 
Therefore the function 
$f:\R\longto \R$
 is a continuous compactly supported real valued function such that
$\mbox{supp}~f\subseteq [0, n_0\mu]$.
 For $ q= p^n$ where $n\geq 1$, we can write
$$ \frac{1}{q^d}\ell(M/I^{[q]}M)=  
\frac{1}{q^d}\sum_{m\geq 0}\ell(M/I^{[q]}M)_{m} =
\quad\quad\quad\quad\quad\quad\quad\quad\quad\quad\quad\quad\quad\quad\quad $$
$$\int_{0}^{1/q}\frac{1}{q^{d-1}}\ell(M_0)dx + \cdots +
 \int_{1-\frac{1}{q}}^{1}\frac{1}{q^{d-1}}\ell(M_{q-1})dx +
\int_{1}^{1+\frac{1}{q}}\frac{1}{q^{d-1}}\ell(\frac{M}{I^{[q]}M})_{q}dx ~~~+
 \quad\quad\quad\quad\quad $$
$$ \int_{1+\frac{1}{q}}^{1+\frac{2}{q}}\frac{1}{q^{d-1}}\ell(\frac{M}{I^{[q]}M})_{q+1}    
dx+ \cdots +
\int_{n_0\mu-\frac{1}{q}}^{n_0\mu}\frac{1}{q^{d-1}}
\ell(\frac{M}{I^{[q]}M})_{n_0\mu q-1}dx
= \int_{0}^{n_0\mu}f_n(x)dx.$$
Therefore $$e_{HK}(M,I) = \lim_{n\to \infty}\int_{0}^{n_0\mu}f_n(x)dx.$$

But, as $\{f_n\}_{n\in \N}$ converges uniformly to $f$, we have 
$$ \lim_{n\to \infty}\int_{0}^{n_0\mu}f_n(x)dx =
 \int_0^{n_0\mu}\lim_{n\to \infty}f_n(x)dx = \int_{\R}f(x)dx .$$
This proves the theorem $\Box $

\vspace{10pt}

Having proved the existence of Hilbert-Kunz density function we are ready to check some 
properties of the function. 

\begin{rmk}\label{r777}Note that (as argued in the proof of the above Theorem)
the  support~$HKd(M, I) \subseteq [0, n_0\mu]$, where 
$n_0$ and $\mu$ are invariants depending on $I$ and $R$, as given in 
Notations~\ref{n1}~part~(3).
\end{rmk}

The first thing we note (Proposition~\ref{r77} below) is that, like HK multiplicity, 
the function 
$$HKd(-, I):\{\mbox{finitely generated graded}~R~\mbox{modules}\}\longto \sC^0_c(\R),$$
is additive, where $\sC^0_c(\R)$ denotes 
the set of continous compactly supported real valued functions. Hence can reduce 
various results  about a HK density function of a module 
to a HK density function of  an integral domain.
Corollary~\ref{cp2}, shows that the 
HKd function is a  multiplicative functor 
with respect to the Segre products on the set of graded $R$-modules.

\begin{propose}\label{r77}(Additive property)~~Let  $R$ be a standard graded ring of dimension $d\geq 2$ over 
a perfect field, and let $I\subset R$ be a homogeneous ideal of finite colength. Let 
$M$ be a finitely generated graded $R$ module. Let $\Lambda $ be the set of 
minimal prime ideals $P$ of $R$ such that $\dim~R/P = \dim~R$. Then 
$$HKd(M, I) = \sum_{P\in\Lambda}HKd(R/P, I)\lambda(M_P).$$ 
\end{propose}
\begin{proof} As usual, there is no loss of generality in assuming that the ground field 
$k$ is algebraically closed. 
Let $\sM$ be the sheaf of $\sO_X$-modules associated to $M$.
For $q=p^n$, recall 
$f_n(M)(x) = \frac{1}{q^{d-1}}\ell(\frac{M}{I^{[q]}M})_{m+q}$, where $\lfloor 
xq\rfloor = m+q$. Let 
${\tilde f_n}(\sM)(x) = \coker~\phi_{m,q}(\sM)/q^{d-1}$.
Note, by Lemma~\ref{r2}~(5), we have 
$$ \lim_{n\to \infty}{\tilde f_n}(M)(x) =  
\lim_{n\to \infty}{f_n}(M)(x)\quad\mbox{for all}\quad x\in \R.$$
Therefore, if $\sQ$ is a coherent sheaf on $X$ then 
we can define $HKd(\sQ, I) := HKd(Q, I)$, where $Q$ is any finitely generated 
graded $R$-module with $\sQ$ as  the associated sheaf of $\sO_X$-modules. 
Note that, due to  Remark~\ref{r777}, 
one can assume $(m+q)^{d-2}\leq (n_0\mu q)^{d-2}$. Therefore, it follows from 
Lemma~\ref{l22} that 
if $M'\longto M''$ is a generic isomorphism of $R$-modules 
({\it{i.e.}}, the kernel and cokernel of the map are of dimension $< \dim~R$) then 
$HKd(M', I) = HKd(M'', I)$. Similarly if
$\sM'\longto \sM''$ is a generic isomorphism of coherent sheaves of 
$\sO_X$-modules  then $HKd(\sM', I) = HKd(\sM'', I)$.

Now let $s\geq 0$ such that $(\mbox{nilradical}~R)^{p^s} = 0$. Define $\sN = F^s_*(\sM)$,
let $q' = p^s$.
Then $\sN$ is a coherent sheaves of $\sO_{X_{red}}$-modules.
Let 
$${C}(h) = \max_{0\leq n_1 <q'}
\left\{C(h_{n_1})\mid h_{n_1}:\sM\longto \sM(n_1)\right\},$$ 
where $h_{n_1}$ is a fixed generically isomorphic map of sheaves of $\sO_X$-modules and 
 $C(h_{n_1})$ is 
the constant (see Lemma~\ref{l22})  
associated to the map $h_{n_1}$ (note that since $k$ is infinite, we can 
always find such a map $h_{n_1}$, for each $n_1$).
Moreover (compare~(\ref{e16})
$$\coker\phi_{m,q}(\sN) = \coker\phi_{mq, qq'}(\sM),\quad\mbox{for all}\quad m\geq 0
\quad\mbox{and}\quad n\geq 0.$$
 
Therefore 
$$|\frac{1}{q'^{d-1}}{\tilde f_n}(\sN)(x)-{\tilde f}_{n+s}(\tilde \sM)|
= \frac{1}{(qq')^{d-1}}|\coker\phi_{mq, qq'}(\sM)-\coker\phi_{mq+n_1, qq'}(\sM)|$$
$$\leq \frac{C(h_n)(m+q)^{d-2}}{(qq')^{d-1}}\leq \frac{C(h)(n_0\mu)^{d-2}}{qq'}.$$
This implies 
$HKd(\sN, I)/(q')^{d-1} = HKd(M,I)$.

Let $Y_1, \ldots, Y_r$ be the irreducible reduced components of $Y = X_{red}$ 
corresponding to the primes ideals in the set 
$\Lambda = \{P_1, \ldots, P_r\}$. Let $x_1, \ldots, x_r$ denote the 
respective generic points in $Y$. Now the canonical generic isomorphism 
$\sN\longto \oplus_i\sN\mid_{Y_i}$
of sheaves of 
$\sO_Y$ (hence $\sO_X$-modules) gives
$$HKd(\sN, I) = \sum_{i=1}^r HKd(\sN\mid_{Y_i}, I).$$

Since $\sN_i = \sN\mid_{Y_i}$ is a coherent sheaf of $\sO_{Y_i}$-modules, there exists 
$a\geq 0$ such that $\sN_i(a)$ is globally generated (Theorem~5.17, Chapter~II in [Ha]), 
for all $i$. 
Hence, 
if $\rank~\sN_{x_i} = \rank~(\sN_{i})_{x_i} = r_i$ as
$\sO_{Y_i, x_i} = \sO_{Y,{x_i}}$-modules then 
there exists a generic isomorphism 
$\oplus^{r_i}\sO_{Y_i} \longto \sN_i(a)$
of $\sO_Y$-modules. Note that $\sN_i$ 
is generically isomorphic to $\sN_i(a)$. Therefore
$$HKd(\sN_i, I)= HKd(\sN_i(a), I) = 
HKd(\sO_{Y_i}, I)\ell(\sN_{x_i})$$
$$\implies HKd(\sN, I) = \sum^r_{i=1} 
HKd(\sO_{Y_i}, I)\ell(\sN_{x_i}) = (p^s)^{d-1}\sum^r_{i=1} 
\ell(M_{P_i})HKd(R/P_i, I).$$
$$\mbox{Therefore}\quad HKd(M, I) = \sum^r_{i=1} 
\ell(M_{P_i})HKd(R/P_i, I).$$ Hence the result.
\end{proof}

\begin{rmk}\label{r7} For $R$ and $I$ as above, in addition suppose 
$R$ equidimensional ring.
and   $I \subseteq J$ are two graded ideals of $R$. Then 
we claim:
$$HKd(R, I) \simeq HKd(R, J)~~\mbox{if and only if}~~ J\subseteq I^{\ast},$$ 
where 
$I^{\ast}$ denotes the {\em tight closure} of $I$ in $R$. To see this, 
 we use the following 
result by   [HH] and [A]: If $(R, {\bf m})$ is a formally unmixed 
local ring with  ${\bf m}$-primary ideals $I\subseteq J$ . Then $e_{HK}(I) =
 e_{HK}(J)$ if and only if $J\subseteq I^{\ast}$. 

Note that in the graded case, the completion ${\hat R}$ of $R$ with 
respect to $R_+$ is an  equidimensional local ring. Also it is easy to 
see that
the tight closure of a graded ideal is a graded ideal.
Now, if  $HKd(I) = HKd(J)$ then by Theorem~\ref{t1}, we have 
 $e_{HK}(I) = e_{HK}(J)$,
 therefore $e_{HK}({\hat I}) = e_{HK}({\hat J})$.
 By [HH] and [A], we have ${\hat J} \subseteq ({\hat I})^{\ast}
 \subset (I^{\ast})^{\wedge}$.
  Hence $J\subseteq I^{\ast}$.
Conversely 
$J\subseteq I^{\ast}$ implies that $e_{HK}(I) = e_{HK}(J)$. But then $HKd(I) \geq 
HKd(J)$ are continuous functions with the same integrals, which implies
$HKd(I) = HKd(J)$.
\end{rmk}

\begin{defn}\label{d2} Similar to the 
  HK density function for pair  $(R, {\bf m})$, where  $R$ is a  
standard graded ring $R$, of $\dim~R\geq 2$,
and ${\bf m}$ is  the graded maximal ideal, we can define the Hilbert-Samuel density 
function  
as 
$$HSd(R)(x) = F(x) = \lim_{n\to \infty}F_n(x),~~~~\mbox{where}~~~F_n(x) = 
\frac{1}{q^{d-1}}\ell(R_{\lfloor xq \rfloor}).$$
One can check that 
 $$F:\R\rightarrow \R~~\mbox{is given by}~~~ 
F(x) = 0,~~~\mbox{for}~~~ x < 0,~~~~\mbox{and}~~~ 
F(x) = e_0(R, {\bf m})x^{d-1}/(d-1)!,~~~~\mbox{ for}~~~ x\geq 0,$$
   where 
$e_0(R,{\bf m})$ is the Hilbert-Samuel multiplicity of $R$ with respect to 
${\bf m}$.

Note  that  
$$HKd(R,I)(x) = HSd(R)(x) = e_0(R, {\bf m})x^{d-1}/(d-1)!,~~~~\mbox{for all}~~~
 x <\min\{n\mid I_n\neq 0\},$$ in particular for all 
$x<1$. 
 \end{defn}

\begin{propose}\label{p2}Let $R_1, \ldots, R_r $ be standard graded 
rings of dimensions $\geq 2$, over an algebraically closed field $k$ of $\mbox{char}~p > 
0 $, with irrelevant maximal ideals ${\bf m}_1, \ldots, {\bf m}_r$ and 
 let  $I_1, \ldots, I_r$ be homogeneous ideals, respectively, such that
 $\ell(R_i/I_i) < \infty$. Let us denote  $HSd(R_i)(x) = {\tilde {F_i}(x)}$ and 
$HKd(R_i, I_i) = {\tilde {f_i}(x)}$. Then 
$$HKd(R_1 \# \cdots \# R_r, I_1 \# \cdots\# I_r)(x) = 
\prod_{i=1}^r{\tilde {F_i}(x)} -
 \prod_{i=1}^r\left({\tilde {F_i}(x)}-{\tilde {f_i}(x)}\right).$$
In particular
$$e_{HK}(R_1\#\cdots \# R_r, I_1\#\cdots\# I_r) = 
\int_0^{n_0\mu}\left\{ \prod_{i=1}^r{\tilde {F_i}(x)} -
 \prod_{i=1}^r\left({\tilde {F_i}(x)}-{\tilde {f_i}(x)}\right)\right\}dx,$$
where
 $R\# S$ denotes the  Segre product of  graded rings $R$ and $S$,
 given by $(R\#S)_n = R_n \tensor S_n$.
\end{propose}
\begin{proof}We prove the case $r = 2$, rest follows by induction.
Let $(R, {\bf m}_1)$ and $(S, {\bf m}_2)$ be two standard graded rings of dimension $d_1$ and $d_2$ 
respectively such that $I$ and $J$ are two  homogeneous ideals of $R$ and $S$
 reply with $\ell(R/I) < \infty $ and $\ell(S/J) < \infty $. Then
$$\ell\left(\frac{R\# S}{(I\# J)^{[q]}}\right)_{m+q} = 
\ell(R_{m+q})\ell(S_{m+q})- 
\left[\ell(R_{m+q})-\ell(\frac{R}{I^{[q]}})_{m+q}\right] 
\left[\ell(S_{m+q})-\ell(\frac{S}{J^{[q]}})_{m+q}\right]$$
$$ =  \ell(R_{m+q})\ell(S/J^{[q]})_{m+q}
+ \ell(S_{m+q})\ell(R/I^{[q]})_{m+q} -  
\ell(R/I^{[q]})_{m+q}\ell(S/J^{[q]})_{m+q}.$$
Let  $F(x)$ and $G(x)$ be HSd functions  of $R$ and $S$ respectively  and 
 let $f(x)$ and $g(x)$ be  HKd functions of $(R, I)$ and $(S,J)$ resply.
Then
$$\frac{1}{q^{d_1+d_2-2}}\ell\left(R\# S/(I\# J)^{[q]}\right)_{\lfloor xq\rfloor} = 
F_n(x)g_n(x) + G_n(x)f_n(x) - f_n(x)g_n(x).$$
If $n_0\geq 1$ is such that ${\bf m}_1^{n_0} \subseteq I$ and
 ${\bf m}_2^{n_0} \subseteq J$, for graded maximal ideal ${\bf m}_1$
 and ${\bf m}_2$
 of $R$ and $S$ resply and $\mu \geq \mu(I)$ and $\mu(J)$ then,  
 by Lemma~\ref{l1}, $F_n(x)g_n(x)$, $G_n(x)g_n(x)$ $f_n(x)g_n(x)$ are 
bounded real valued function with support in the interval $[0, n_0\mu]$.
Moreover, by Theorem~\ref{t1}, $f_n(x)$ and $g_n(x)$ converge 
uniformly to $f(x)$ and $g(x)$ resply. It is obvious that, 
on any compact interval,  
the functions $F_n(x)$ and $G_n(x)$ 
converge uniformly to $F(x)$ and $G(x)$ reply.
Therefore  $F_n(x)g_n(x) + G_n(x)f_n(x) - f_n(x)g_n(x)$
 converge uniformly to $F(x)G(x)+G(x)f(x)-f(x)g(x)$ and 
$$HKd~(R\# S, I\# J) = F(x)g(x) + G(x)f(x) - f(x)g(x).$$
This implies that 
$$e_{HK}(R\# S. I\# J) = \frac{e_0(R)}{(d_1-1)!}
\int_0^{n_0\mu}x^{d_1-1}g(x)dx + \frac{e_0(S)}{(d_2-1)!}
\int_0^{n_0\mu}x^{d_2-1}f(x)dx  
-\int_0^{n_0\mu}f(x)g(x)dx.$$

This proves the proposition.
\end{proof}

\begin{cor}\label{cp2}{(Multiplicative property)}.~~For pairs $(R, I)$ and
$(S, J)$ with $\mbox{dim}~R = d_1$ and $\mbox{dim}~S = d_2$, if 
$F(x)$ and $G(x)$  denote HSd functions  of $R$ and $S$ respectively  as 
given in Definition~\ref{d2}
then we have 
$$F_{R\#S}- HKd(R\# S, I\#J) = \left[F_R-HKd(R,I)\right]\cdot\left[F_S-HKd(S, J)\right].$$
\end{cor}
\begin{proof}Follows from Proposition~\ref{p2}.
\end{proof}
\vspace{5pt}

\begin{thm}\label{t*1}Let $R$ be a standard graded reduced ring of 
dimension $1$ and 
$I$ be a  homogeneous ideal of $R$ such that $\ell(R/I) < \infty $. 
Let $f_n(x) = \ell(R/I^{[q]})_{\lfloor xq\rfloor} $.
 Then $\{f_n(x)\}_{n\in \N}$ is a convergent (but need not be uniformly
convergent) sequence, 
for every $x\in [0, \infty)$  and  $$e_{HK}(I, R) = \int_{\R}f(x)dx,$$
where  $f(x) = \lim_{n\to \infty}f_n(x)$. 
\end{thm}
\begin{proof} Let $h_1, \ldots, h_{\mu}$ be a set of homogeneous generators of $I$ of degree
$d_1,\ldots, d_{\mu}$ such that $d_0 =0  < d_1 \leq d_2\leq \ldots \leq d_{\mu}$. 

Since $R$ is a $1$ dimensional ring there exists an integer 
 $m_0\geq 1$ such that 
 $\ell(R_m) = \ell(R_{m+1})$, for all $m\geq m_0$.
For $n\in \N$, we define
$$T_n = \left(0, m_0/q\right]\cup (d_1, d_1+m_0/q]\cup \cdots 
\cup (d_{\mu}, d_{\mu}+m_0/q]\subseteq (0, \infty].$$

\noindent{Claim}.\quad 
If $x \notin T_n $ then $f_n(x) = f_{n+l}(x)$, for all
$l\geq 0$.

\vspace{5pt}

\noindent{Proof of the claim}: Since $T_{n+l}\subseteq T_n$, for all $l\geq 0$, 
it is enough to prove that $x\notin T_n$ implies 
 $f_n(x) = f_{n+1}(x)$.
Note that $x\nin T_n$ then 
$$m = \lfloor xq\rfloor \notin 
 (0, m_0)\cup (d_1q, d_1q+m_0) \cup \cdots  \cup (d_{\mu}q, d_{\mu}q +m_0).$$
By definition 
$$f_n(x) = \ell(R/I^{[q]})_{m}~~~\mbox{and}~~~
 f_{n+1}(x) = 
\ell(R/I^{[qp]})_{mp+n_1},$$
where $\lfloor xpq\rfloor = \lfloor xq\rfloor p+ n_1$, for some  
$0\leq n_1 < p$.
Choose a nonzero divisor $a\in R_1$. Then we have the injective map
$R_m \rightarrow R_{mp+n_1}$ given by $y\mapsto a^{n_1}y^p$ (this is a
 composition of two maps namely
 $R_{m} \rightarrow R_{mp}$, given by $y\mapsto y^p$, and 
$R_{mp} \rightarrow R_{mp+n_1}$, given by $x\mapsto a^{n_1}x$)
which is an isomprphism (as $k$-vectorspaces) as $m  =
\lfloor xq \rfloor  \geq m_0$. This gives a canonical
 surjective map
$ \phi: (R/I^{[q]})_m \longrightarrow(R/I^{[qp]})_{mp+n_1}$.
Now to prove the claim, it is enough to prove that 
 $(I^{[qp]})_{mp+n_1} \subseteq \phi(I^{[q]}_m)$. 
 Let $f\in (I^{[qp]})_{mp+n_1}$
 then $f = h_1^{qp}r_1+\cdots +h_{\mu}^{qp}r_{\mu}$, where $\deg~r_j = 
mp+n_1-d_jqp$.

If $r_j \neq 0 \implies mp+n_1-d_jqp \geq 0 \implies 
m-d_jq\geq -n_1/p \implies    m-d_jq \geq 0 \implies 
xq \geq d_jq$.
\begin{enumerate}
\item $xq = d_jq $ then $n_1 = 0$ and $mp-d_jqp = 0$. Therefore $r_j\in R_0 = k$.
 Hence $r_j = l_j^p$, for some $l_j \in R_0$.
\item $xq > d_jq  \implies m \geq d_jq+m_0$. but then $\deg~r_j 
\geq mp - d_jqp+n_1 \geq m_0p+n_1$. Therefore $r_j = l_j^pa^{n_1}$, 
for some $l_j \in d_jq+m$. 
\end{enumerate}

This implies  
$f 
 = (h_1^{q}l_1+\cdots + h_{\mu}^{q}l_{\mu})^pa^{n_1} \in  
\phi(I^{[q]}_{m})$.
This proves the claim.

 Define $f(x) = \lim_{n\to \infty}f_n(x)$; this makes sense because 
\begin{enumerate}
\item if $x = 0$ then $f_n(0) = \ell(R_0)$, for all $n$.
\item If $ x> 0$ then there exists $n >0$ such that $x\notin T_n$, which
 implies that  
$f_n(x) = f_{n+1}(x) = \cdots = f(x)$.\end{enumerate}
Moreover, for $y\in \R$, $f_n(y) \leq L_2(R)$, where  
$L_2(R) = \mbox{max}\{ \ell(R_0), \ell(R_1), \ldots, 
\ell(R_{m_0})\}.$
Therefore we have 
$$|\int_{\R}f_n(x)dx-\int_{\R}f(x)dx| \leq \int_{\R}|f_n(x)-f(x)|dx \leq 
\int_{T_n}|f_n(x)-f(x)|dx \leq L_2(R)(\mu+1)m_0/q.$$
Hence 
$$\lim_{n\to \infty}\int_{\R}f_n(x)dx = \int_{\R}f(x)dx = 
\int_{\R}\left(\lim_{n\to \infty}f_n(x)\right)dx.$$\end{proof}

\begin{rmk}\label{r3}It is easy to check that in the case of dimension $1$,
 $f_n\rightarrow f$ does not converge to $f$ uninformly.
\end{rmk}

\section{Examples}
\subsection{Projective spaces and their Segre products}
\begin{ex}\label{ex4}Let $X= {\mathbb P}^d_k$ and let $R= k[X_0, 
\ldots, X_d] = \oplus_mR_m$.
We denote the function $HKd(R, {\bf m})$ by $HKd(\P^d_k)$ and for a fixed 
$q = p^n$, we denote the map $\phi_{m, q}(R)$ by 
$\phi_m$ where we recall that $\phi_{m, q}(R): R_1^{[q]}\tensor R_m \rightarrow
 R_{m+q}$ is the canonical  multiplication map. 
 For $A_m= {{m+d}\choose{d}}$, it is obvious that   
 $$\coker~{\phi}_{tq+l} = A_{(t+1)q+l}-A_1\coker~\phi_{(t-1)q+l}+
A_2\coker~\phi_{(t-2)q+l}+\cdots + (-1)^{t+1}A_{t+1}\coker~\phi_{l-q}.$$

Now, for $q=p^n$,  
$$f_n(x) = \coker\phi_{tq+l},~~\mbox{ where}~~~ 
\frac{(t+1)q+l}{q} \leq x < \frac{((t+1)q+l+1)}{q}~~~\mbox{ with}~~~ 0\leq l 
<q.$$
Hence 
 $$f_n(x) = ({1}/{q^d})A_{(t+1)q+l}-A_1f_n(x-1)+\cdots+
(-1)^{t+1}A_{t+1}f_n(x-t-1).$$ 
Moreover $f_n(x) = 0$, if $x\geq d+1$.
Therefore
$$HKd(\P^d_k)(x) = f(x) = \frac{1}{d!}\left[x^d-{\tilde A_1}(x-1)^d + 
{\tilde A_2}(x-2)^d + \cdots
+ (-1)^{t+1}{\tilde A_{t+1}}(x-t-1)^d\right],$$
where ${\tilde A_1} = (d+1)$ and ${\tilde A_2} = {{d+1}\choose{2}}$ and 
${\tilde A_{i+1}}$ are defined iteratively as  
$${\tilde A_{i+1}}= A_1{\tilde A_i}-A_2{\tilde A_{i-1}}+ \cdots +
y(-1)^{i-1}A_{i}{\tilde A_1}+(-1)^iA_{i+1}.$$
This implies 
${\tilde A_i} = {{d+1}\choose{i}}$, for $1\leq i\leq d$.
In particular
$$\begin{array}{lcl}
HKd(\P^d_k)(x) & = & x^d/d!,~~~\mbox{for}~~0\leq x <1\\
 & = & x^d/d!- A_i^d(x),~~~~~~\mbox{for}~~~~i\leq x 
< i+1\quad\quad~\mbox{provided}\quad
1\leq i\leq d\\
& = & 0 \quad\quad\mbox{otherwise},\end{array}
$$
where 
$$A_i^{d}(x) = \frac{1}{d!}\left[{{d+1}\choose{1}}(x-1)^d+\cdots +(-1)^{i+1}   
{{d+1}\choose{i}}(x-i)^d\right].$$
 Moreover $HSd(\P^d_k)(x) = x^d/d!$.

Therefore for the Segre product $\P^d_k\#\P^e_k$, where $d\leq e$,  
we have
$$\begin{array}{lcl}
HKd(\P^d_k\#\P^e_k)(x) & = & \displaystyle{\frac{x^dx^e}{d!e!} - 
A^d_i(x)A^e_i(x)}\quad\quad~~~~~~\mbox{for}~~\quad i\leq x < i+1,
\quad\mbox{provided}\quad 1\leq i\leq d\\
 & = & \displaystyle{\frac{x^dx^e}{d!e!}-\frac{x^d}{d!}A^e_i(x)},
\quad\quad~~~~~\mbox{for}\quad\quad~~d\leq x < e\\
&  = & 0,\quad\quad\mbox{for}~~\quad e\leq x.\end{array}$$
\end{ex}

\begin{rmk}Similar (but more complicated) formulas can be obtained for arbitrary Segre 
products of projective spaces. 
The Hilbert-Kunz multiplicity of the Segre product of
 $\P_k^n\times \P_k^m$ has been computed by [EY].
\end{rmk}

\subsection{Projective curves and their Segre products}

\begin{ex}\label{ex1}Let $R$ be a Noetherian standard graded ring of 
dimension $2$.
 Then, for a pair $(R, {\bf m})$, where ${\bf m}$ is the graded maximal ideal,
 $e_{HK}(R, {\bf m})$ has
 been computed in [B] and [T1], Here  we compute  
$HKd(R, {\bf m}) = f:\R \longto \R$ using the  similar techniques  used in
these two papers. 

Recall that if $x\in [0, 1)$ then
 $$ f_n(x) = \frac{1}{q} \ell(R_m),~~~\mbox{where}~~ m/q\leq x < m+1/q.$$
 This implies that $HKd(R, {\bf m})(x) = f(x) = \lim_{n\mapsto \infty}
 f_n(x) = (d)(x)$, 
where $d := e_0(R,{\bf m})$ is the Hilbert-Samuel multiplicity of $R$ with
 respect to the graded maximal ideal ${\bf m}$.

Now let $1\leq x$ then $(m+q)/q\leq x < (m+q+1)/q$, for some $m\geq 0$, and
$$f_n(x) = \frac{1}{q}\ell(R/{\bf m}^{[q]})_{m+q} = 
\frac{1}{q}\ell(R/{\bf m}^{[q]})_{\lfloor xq\rfloor}.$$
Let $h_1, \ldots, h_s\in R_1$ be a set of generators of ${\bf m}$ and let 
$$0\longto V\longto \oplus\sO_X\longto \sO_X(1)\longto 0$$
be the map of locally free sheaves of $\sO_X$-modules.
By Lemma~\ref{r1},~Part(2), it follows that 
$$HKd(R, {\bf m})(x)= f(x) = \lim_{n\mapsto \infty}\frac{1}{q}h^1(X, F^{n*}V({\lfloor (x-1)q\rfloor})).$$
By Theorem~2.7 in [L], there exists $n_1 >>0$ such that 
$$ 0 = E_0 \subset E_1 \subset \cdots \subset E_l \subset E_{l+1} \subset 
F^{n_1*}V$$ 
is the strong Harder-Narasimhan filtration of $F^{n_1*}V$.
Let 
\begin{equation}\label{e14}
a_i= \mu_i(F^{n_1*}V)/ p^{n_1} = \mu(E_i/E_{i-1})/p^{n_1}, 
~~~~r_i = \rank(E_i/E_{i-1})\end{equation}
 be the normalized HN slope of $V$.
 Note that $a_i's$ are independent of the choice of $n_1$ as 
$\mu_i(F^{n*}V) = p^{n-n_1}\mu_i(F^{n_1*}V)$, for all $n\geq n_1$.
Since $V\into \oplus\sO_X$, 
$a_i(V) \leq 0$.
In fact 
$$-\frac{a_1}{d}  < -\frac{a_2}{d} <  \cdots < -\frac{a_{l+1}}{d}.$$

Moreover,
we can take $n_1>>0$ such that 
$$\mu_i(F^{n_1*}V)-\mu_{i+1}(F^{n_1*}V) \geq 2g-2.$$
Therefore $$h^1(X, F^{n*}V(m)) = \sum_{i=1}^{l+1}h^1(X, F^{n-n_1*}(E_i/E_{i-1})(m)).$$
 and 
$$-\frac{a_1q}{d}  < -\frac{a_1q}{d} +(d-3) < -\frac{a_2q}{d} <-\frac{a_2q}{d} +(d-3)
< \cdots < -\frac{a_{l+1}q}{d}.$$
Hence, we have
$$\begin{array}{lcl}
 0 \leq m < -\frac{a_1q}{d} & \implies & f_n(x) 
= - \frac{1}{q} \sum_{i = 1}^{l+1} 
 (a_iqr_i+r_idm + r_i(g-1))\\ 
-\frac{a_iq}{d} \leq m < -\frac{a_iq}{d} + (d-3) & \implies & 
f_n(x) = -\frac{1}{q}\sum_{j = i+1}^{l+1} 
 (a_jqr_j+r_jd{m} + r_j(g-1)) + \frac{C_i}{q}\\ 
-\frac{a_iq}{d} \leq m < -\frac{a_{i+1}q}{d} & \implies & 
f_n(x) = -\frac{1}{q} \sum_{j = i+1}^{l+1} 
 (a_jqr_j+r_jd{m} + r_j(g-1)),
\end{array}$$
where $|C_i| \leq r_i(g(X)-1)$. 

Therefore 
 $$\begin{array}{lcl}
1 \leq x < 1-a_1/d & \implies & f(x) = -\sum_{i = 1}^{l+1} (a_ir_i+r_id(x-1))\\ 
1-a_1/d \leq x < 1-a_2/d & \implies & f(x) = 
-\sum_{i = 2}^{l+1} (a_ir_i+r_id(x-1))\\ 
1-a_i/d  \leq x < 1-a_{i+1}/d  & \implies & f(x) = 
- \sum_{j = i+1}^{l+1} (a_jr_j+r_jd(x-1))
\end{array}$$

This implies $$e_{HK}(R, {\bf m}) = \int_{x = 0}^{1-a_{l+1}/d}f(x)dx =
 \int_{x = 0}^{1}f(x)dx + \int_{x = 1}^{1-a_1/d}f(x)dx + \cdots + \int_{x = 
1-a_l/d}^{1-a_{l+1}/d}f(x)dx$$
 $$ = d/2 - \int_{y = 0}^{-a_1/d}[a_1r_1 + 
(r_1d)y]dy - \int_{y = 0}^{-a_2/d}[a_2r_2 + (r_2d)y]dy + \cdots - \int_{y = 
0}^{-a_{l+1}/d}[a_{l+1}r_{l+1} + (r_{l+1}d)y]dy $$ Therefore 
$$ e_{HK}(R, {\bf m}) 
= \frac{d}{2} + \sum_{i =1}^{l+1}\frac{r_ia_i^2}{2d}.$$
\end{ex}

\begin{rmk} As $\{a_i\}_i$ are distinct numbers, the above 
formula for $f$ implies 
that $HKd(R, {\bf m})$ determines the data $(d, \{r_i\}_i, \{a_i\}_i)$.

Moreover, for a pair $(R, I)$, where $I$ is a graded ideal generated by
 homogeneous elements $h_1, \ldots, h_{\mu}$ of degrees $d_1\leq  \cdots
\leq  d_{\mu}$ respectively,
$$HKd(R, I)(x)= f(x)
 = \lim_{n\mapsto \infty}\frac{1}{q}h^1(X, F^{n*}V({\lfloor (x-1)q\rfloor}))
-  \lim_{n\mapsto \infty}\frac{1}{q}\sum_{i =1}^{\mu}
h^1(X,\sO_X({\lfloor xq\rfloor}-qd_i)).$$
It is easy to check that the second term
 is a piecewise linear polynomial 
(with rational coefficients) with singularities at distinct points of the set
$\{d_1, \ldots, d_{\mu}\}$ and support in $[0, d_{\mu}]$.
In particular, there exists rational numbers
$0 = x_0 < x_1 <\cdots < x_s \leq \max\{n_0\mu, d_{\mu}\}$
and linear polynomials  $q_i(x)\in \Q[x]$,
such that 
 $HKd(R,I)(x)  = q_i(x)$ if $x\in 
[x_i, x_{i+1}]$ and $HKd(R, I)(x)  =  0$ otherwise.
\end{rmk}

For the following corollary, we use the notation of Proposition~\ref{p2}.

\begin{cor}\label{c1}Any Segre product of projective curves has rational 
Hilbert-Kunz multiplicity. More precisely 
$e_{HK}(R_1\#\cdots \# R_r, I_1\#\cdots\# I_r)$ is a rational number,
 where $\dim~R_i = 2$, for each $i$.
\end{cor}
\begin{proof} Let $n_0, \mu\geq 1$ such that ${\bf m}_i^{n_0}\subseteq I_i$ and 
$\mu \geq \mu(I_i)$, for all $i$, where ${\bf m}_i$ denotes the graded maximal
 ideal of $R_i$. Let ${\tilde d}$ be the maximum of the 
degree of the chosen generators of $I_i$ and $n_0\mu$. 
Now, the above calculation shows that, 
one can take a finite subdivision of the interval $[0, {\tilde d}]$
by rational points $t_i$, namely
$$ [0, {\tilde d}] = \bigcup_{1\leq i\leq m} [t_i, t_{i+1}],~~~~\mbox{where}~~~
t_1< t_2 <\ldots < t_m$$
such that each function $HKd(R_i, I_i)(x)$, on each such interval $[t_j, t_{j+1}]$,
 is a linear polynomial in $\Q[x]$. 
Note that each $HSd(R_i)(x)$  is a polynomial in $\Q[x]$ on whole of $[0, {\tilde d}]$.
 Therefore the assertion follows from Proposition~\ref{p2}.
\end{proof}

\begin{rmk}\label{r4} 
By the same reasoning, as in the above corollary, one can 
prove that the Hilbert-Kunz multiplicities  of  
the arbitrary Segre product of full flag varieties, $\P^n_k$, Hirzebruch surfaces
and projective curves etc. (over a fixed algebraically closed field $k$) are rational numbers.
\end{rmk}

\begin{ex}\label{ex2} Let $(R, {\bf m}_1)$ and $(S, {\bf m}_2)$ be
 two standard graded rings of 
dimension~$2$ with 
graded maximal ideals ${\bf m}_1$ and ${\bf m}_2$ respectively.
 Let $V_1$ and $V_2$ be corresponding syzygy bundles for $(R,{\bf m}_1)$ and 
$(S,{\bf m}_2)$
 with normalized slopes $a_1, a_2, \ldots, a_{i_3}$ and $b_1, \ldots, b_{j_2}$,
repectively.
 Let $X = Proj~R$ and $Y = Proj~S$ be of degree $d$ and $g$ resply.
If 
$$-\frac{a_1}{d}  < -\frac{a_2}{d} <  \cdots < -\frac{a_{i_1}}{d}
\leq -\frac{b_1}{g}  < -\frac{b_2}{g} <  \cdots < -\frac{b_{j_1}}{d}
\leq -\frac{a_{i_1+1}}{d}  <   \cdots < -\frac{a_{i_2}}{d}$$
 $$\leq 
-\frac{b_{j_1+1}}{g}  < -\frac{b_{j_1+2}}{g} <  \cdots < -\frac{b_{j_2}}{d}
\leq -\frac{a_{i_2+1}}{d}  <   \cdots < -\frac{a_{i_3}}{d},$$
then, for $x_i = a_i/d$ and $y_i = b_i/g$, we find (after some computation)
that 
$$e_{HK}(R\# S, {\bf m}_1\# {\bf m}_2) = \displaystyle{\frac{dg}{6} 
\left[2+ \sum_{j\geq 1}3s_jy_j^2-\sum_{j\geq 1} s_jy_j^3 +
\sum_{i\geq 1}3r_ix_i^2 - 
\sum_{i\geq 1}r_ix_i^3\right.}$$
$$+ \displaystyle{(\sum_{j\geq 1}3s_jy_j)(r_1x_1^2+\cdots r_{i_1}x_{i_1}^2) -
(\sum_{j\geq 1}s_j)(r_1x_1^3+\cdots r_{i_1}x_{i_1}^3)}$$
$$+\displaystyle{(\sum_{i\geq i_1+1}3r_ix_i)(s_1y_1^2 + \cdots + 
s_{j_1}y_{j_1}^2) -
(\sum_{i \geq i_1+1}r_i)(s_1y_1^3 + \cdots + s_{j_1}y_{j_1}^3)}$$
$$+\displaystyle{(\sum_{j\geq j_1+1}3s_jy_j)(r_{i_1+1}x_{i_1+1}^2+\cdots + 
r_{i_2}x_{i_2}^2)
-(\sum_{j\geq j_1+1}s_j)(r_{i_1+1}x_{i_1+1}^3+\cdots + r_{i_2}x_{i_2}^3)}$$
$$\displaystyle{\left.+(\sum_{i\geq i_2+1}3r_ix_i)(s_{j_1+1}y_{j_1+1}^2 + 
\cdots + s_{j_2}y_{j_2}^2) -
(\sum_{i\geq i_2+1}r_i)(s_{j_1+1}y_{j_1+1}^3 +
 \cdots + s_{j_2}y_{j_2}^3)\right]}$$
Note that every term of the first row is nonnegative, whereas, 
in the  second to last rows, all the first terms  are nonpositive and all the 
second terms
 are nonnegative.
\end{ex}

\end{document}